\documentclass[10pt]{article}

\usepackage[margin=1in]{geometry} 
\usepackage{pxfonts} 

\usepackage[english]{babel} 
\usepackage{amsmath, amsfonts, amsthm, amssymb} 
\usepackage[scr=boondoxo,cal=euler]{mathalfa} 
\usepackage{tikz-cd} 
\usepackage{indentfirst} 
\usepackage{fancyhdr} 
\usepackage{graphicx} 
\usepackage{enumitem} 
\usepackage{microtype} 
\usepackage{pdfpages} 
\usepackage{centernot} 
\usepackage{ragged2e} 
\usepackage{pinlabel}
\usepackage{caption}
\usepackage[hidelinks]{hyperref} 

\allowdisplaybreaks 
\usetikzlibrary{decorations.pathmorphing}

\theoremstyle{plain}
\newtheorem{thm}{Theorem}[section]
\newtheorem*{thm*}{Theorem}
\newtheorem{lem}[thm]{Lemma}
\newtheorem*{lem*}{Lemma}

\newtheorem*{cor*}{Corollary}
\newtheorem{prop}[thm]{Proposition}
\newtheorem*{prop*}{Proposition}

\newtheorem*{conj*}{Conjecture}

\newtheorem*{ques*}{Question}

\theoremstyle{definition}
\newtheorem{df}[thm]{Definition}
\newtheorem*{df*}{Definition}
\newtheorem*{dfs*}{Definitions}

\newtheorem*{exercise*}{Exercise}

\theoremstyle{remark}
\newtheorem{rem}[thm]{Remark}
\newtheorem*{rem*}{Remark}

\newtheorem{example}[thm]{Example}
\newtheorem*{example*}{Example}

\usepackage{etoolbox}
\patchcmd{\thmhead}{(#3)}{#3}{}{}
\makeatletter
\g@addto@macro\bfseries{\boldmath}
\makeatother

\newcommand{\fk}[1]{\mathfrak{#1}}

\newcommand{\sr}[1]{\mathscr{#1}}
\newcommand{\Z}{\mathbf{Z}} 
\newcommand{\Q}{\mathbf{Q}} 
\newcommand{\R}{\mathbf{R}} 
\newcommand{\C}{\mathbf{C}} 
\newcommand{\G}{\mathbf{G}} 
\newcommand{\F}{\mathbf{F}}
\newcommand{\CP}{\mathbf{CP}} 

\newcommand{\x}{\times}

\newcommand{\emp}{\emptyset}

\renewcommand{\sl}{\fk{sl}}

\newcommand{\qbinom}[2]{\genfrac{[}{]}{0pt}{}{#1}{#2}}
\newcommand{\ol}[1]{\overline{#1}}

\DeclareMathOperator{\rk}{rk}

\DeclareMathOperator{\Hom}{Hom}

\DeclareMathOperator{\Id}{Id}

\DeclareMathOperator{\SU}{SU}
\DeclareMathOperator{\U}{U}

\DeclareMathOperator{\SO}{SO}

\DeclareMathOperator{\KR}{KR}

\usetikzlibrary{decorations.markings}
\usetikzlibrary{shadings}

\tikzset{
	neg/.style={postaction=decorate,
	decoration={markings,
	mark=at position 0.05cm with \node{$\circ$};
	}}
}

\newcommand{\thickRes}{
\begin{gathered}
	\includegraphics[width=.035\textwidth]{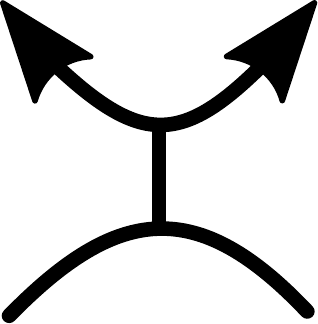}
	\vspace{-3pt}
\end{gathered}
}

\newcommand{\oriRes}{
\begin{gathered}
	\includegraphics[width=.035\textwidth]{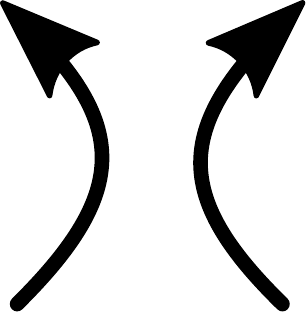}
	\vspace{-3pt}
\end{gathered}
}

\newcommand{\zip}{
\begin{gathered}
	\includegraphics[width=.06\textwidth]{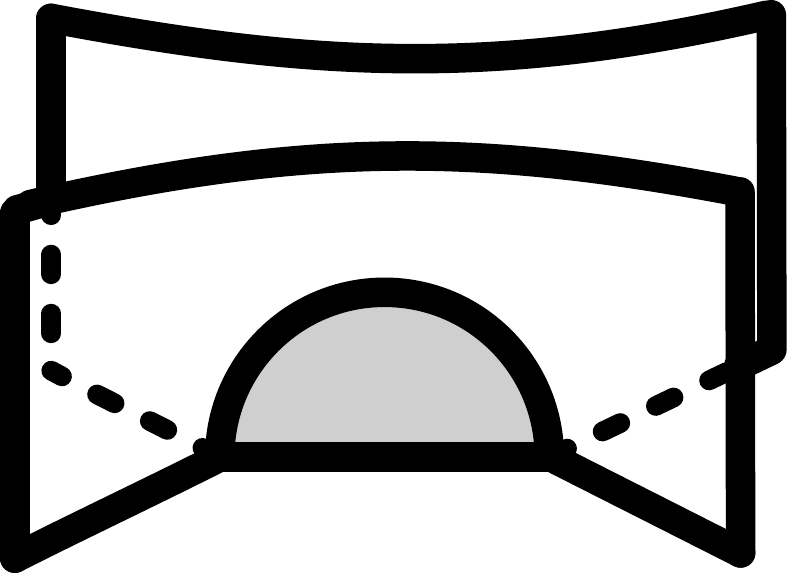}
	\vspace{-5pt}
\end{gathered}	
}

\newcommand{\TRdot}{
\begin{gathered}
	\includegraphics[width=.035\textwidth]{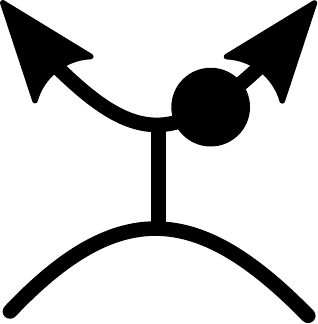}
	\vspace{-3pt}
\end{gathered}
}

\newcommand{\TLdot}{
\begin{gathered}
	\includegraphics[width=.035\textwidth]{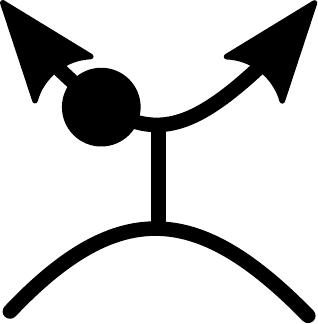}
	\vspace{-3pt}
\end{gathered}
}

\newcommand{\BRdot}{
\begin{gathered}
	\includegraphics[width=.035\textwidth]{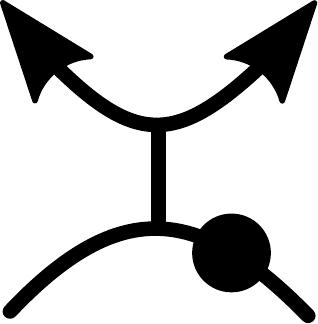}
	\vspace{-3pt}
\end{gathered}
}

\newcommand{\BLdot}{
\begin{gathered}
	\includegraphics[width=.035\textwidth]{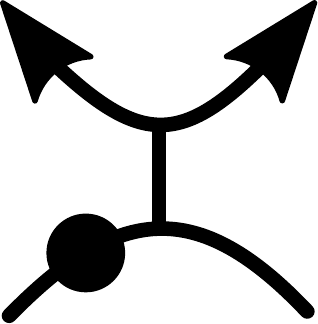}
	\vspace{-3pt}
\end{gathered}
}

\newcommand{\thickLoop}{
\begin{gathered}
	\includegraphics[width=.036\textwidth]{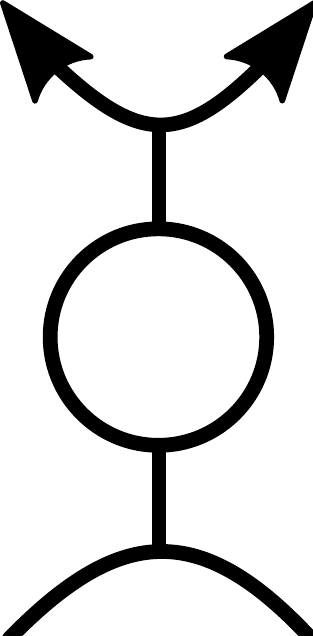}
	\vspace{-3pt}
\end{gathered}	
}

\newcommand{\loopTRdot}{
\begin{gathered}
	\includegraphics[width=.035\textwidth]{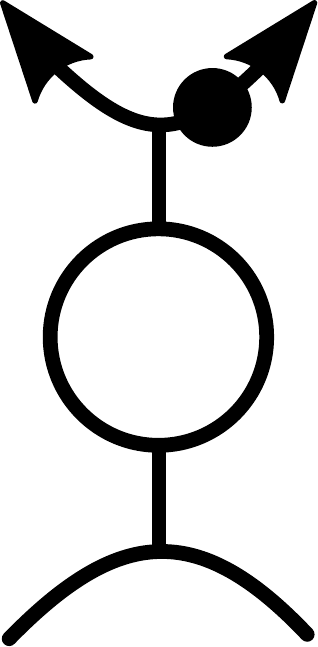}
	\vspace{-3pt}
\end{gathered}
}

\newcommand{\loopBRdot}{
\begin{gathered}
	\includegraphics[width=.035\textwidth]{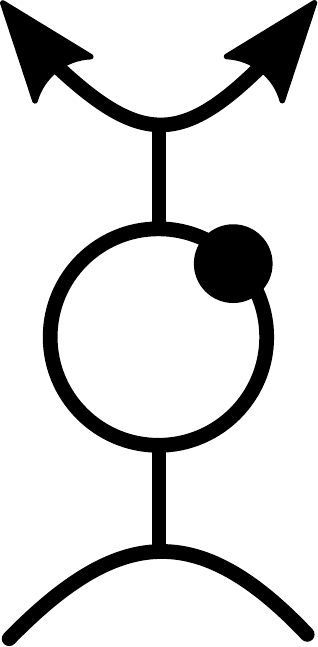}
	\vspace{-3pt}
\end{gathered}
}

\newcommand{\loopBLdot}{
\begin{gathered}
	\includegraphics[width=.035\textwidth]{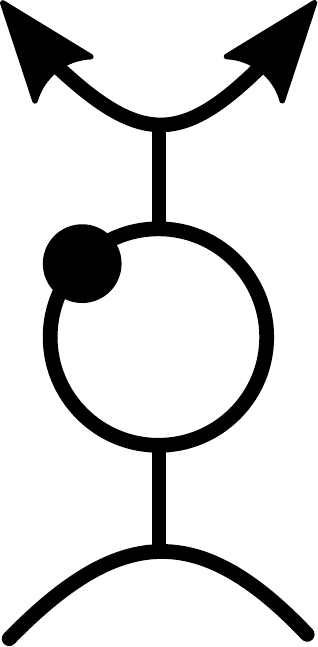}
	\vspace{-3pt}
\end{gathered}
}

\newcommand{\leftZip}{
\begin{gathered}
	\includegraphics[width=.12\textwidth]{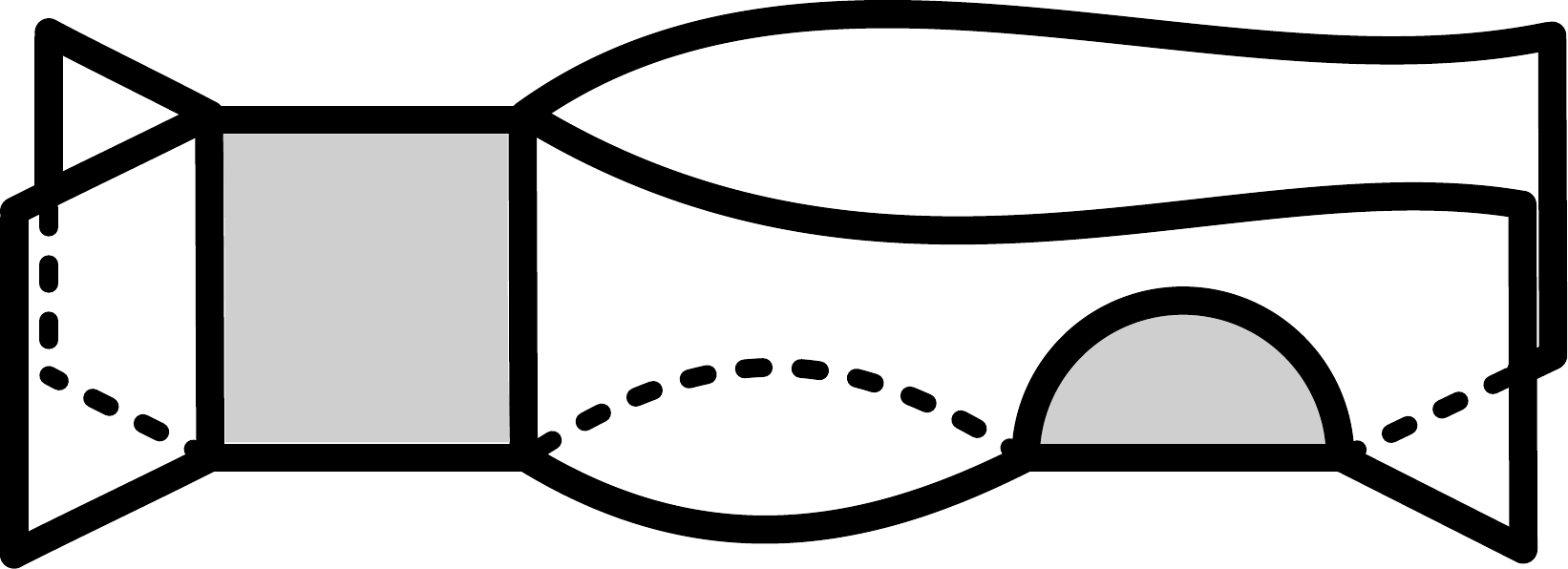}
	\vspace{-3pt}
\end{gathered}	
}

\newcommand{\rightZip}{
\begin{gathered}
	\includegraphics[width=.12\textwidth]{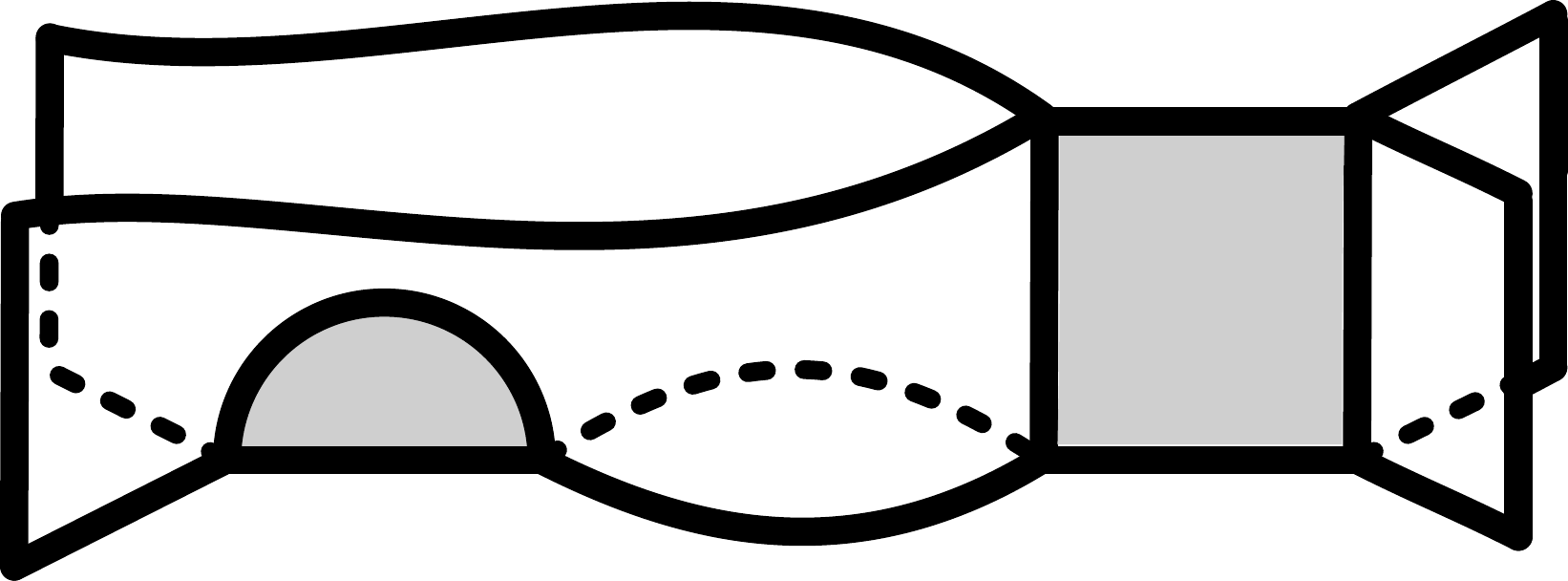}
	\vspace{-3pt}
\end{gathered}	
}

\newcommand{\cupFoam}{
\begin{gathered}
	\includegraphics[width=0.055\textwidth]{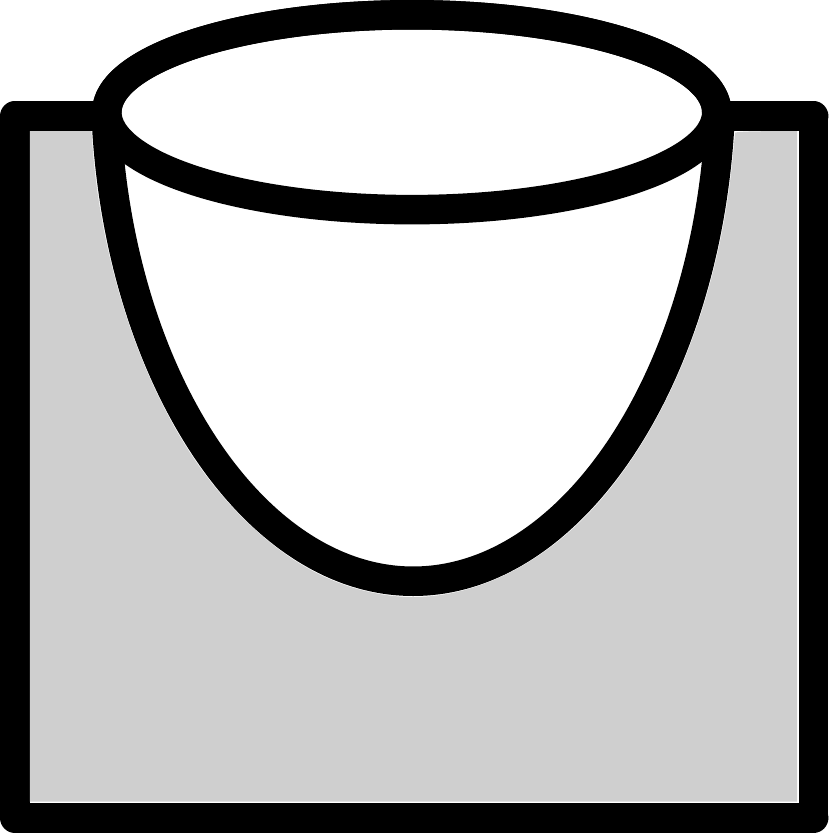}
	\vspace{-3pt}
\end{gathered}
}

\newcommand{\capFoam}{
\begin{gathered}
	\includegraphics[width=.055\textwidth]{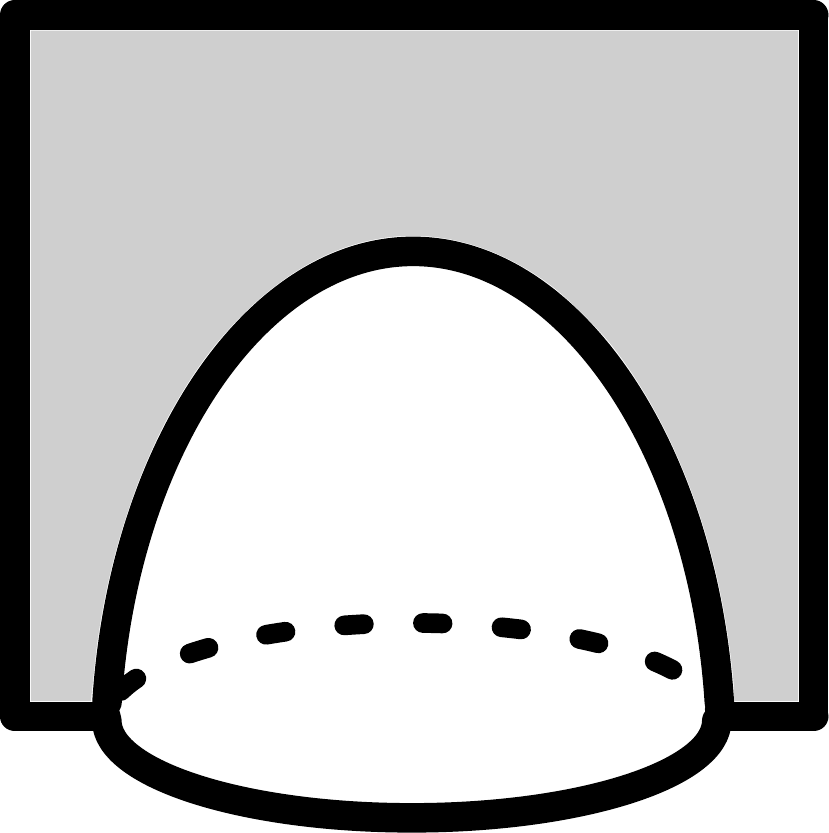}
	\vspace{-3pt}
\end{gathered}
}

\newcommand{\cupEx}{
\begin{gathered}
	\includegraphics[width=.12\textwidth]{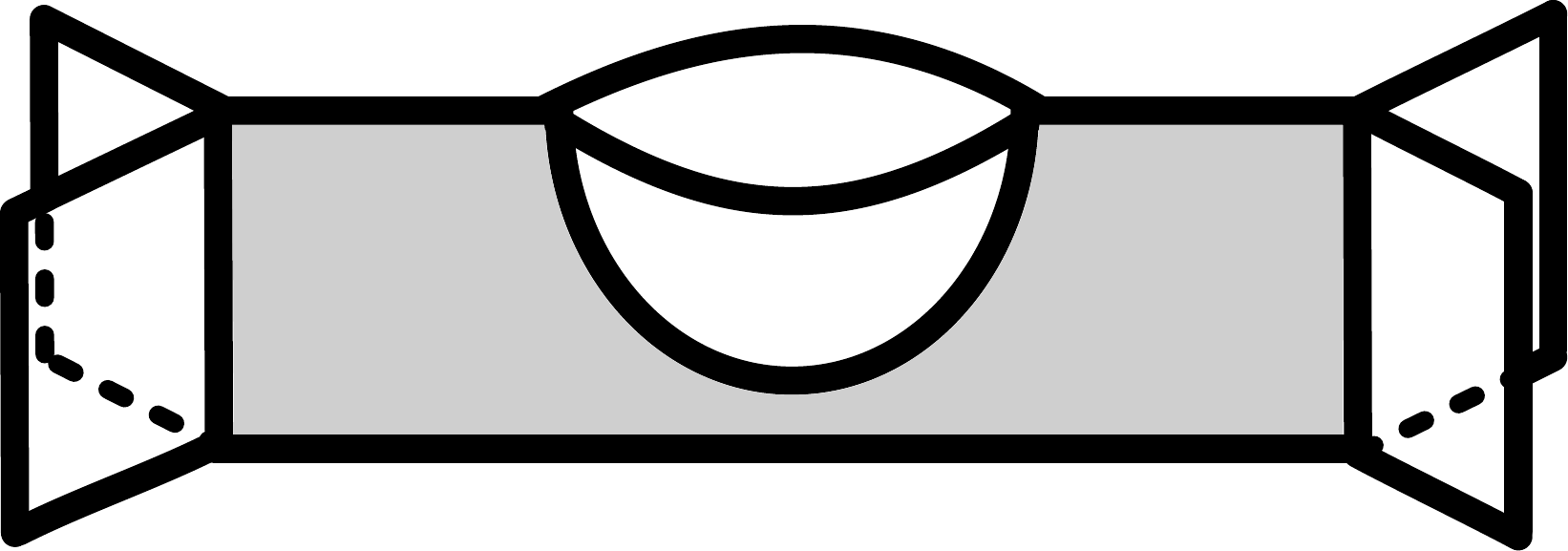}
	\vspace{-3pt}
\end{gathered}
}

\newcommand{\capEx}{
\begin{gathered}
	\includegraphics[width=.12\textwidth]{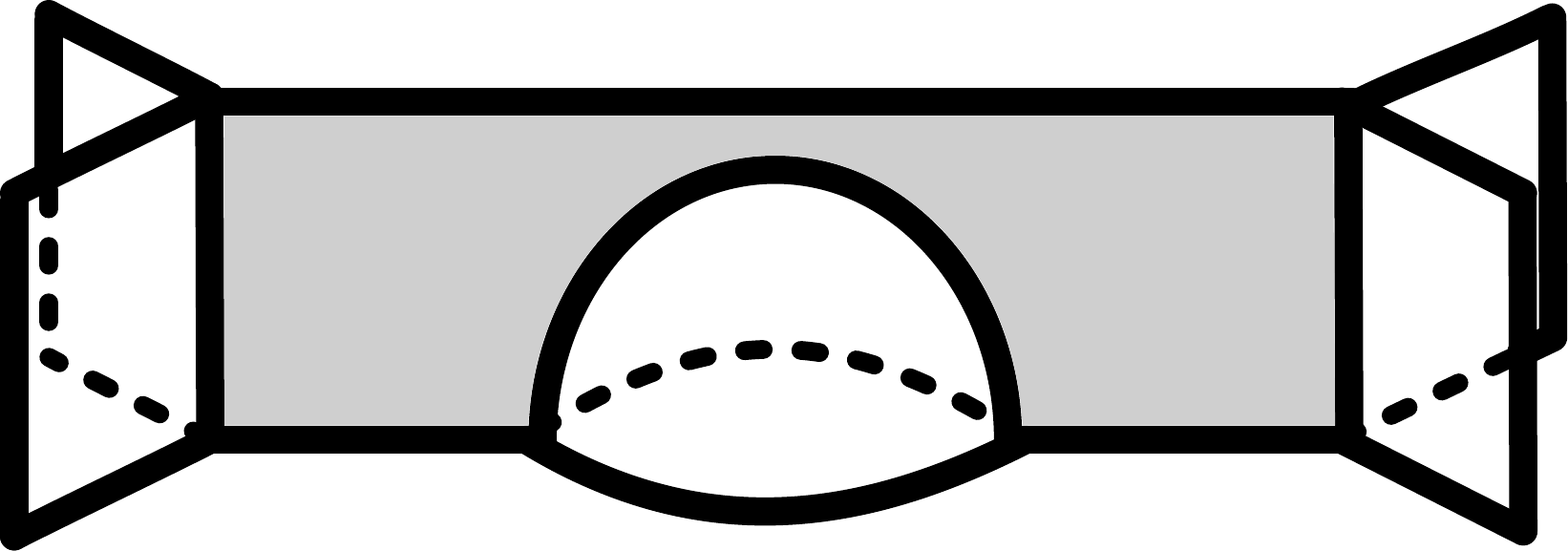}
	\vspace{-3pt}
\end{gathered}
}

\newcommand{\Ldot}{
\begin{gathered}
	\vspace{-3pt}
	\includegraphics[width=.025\textwidth]{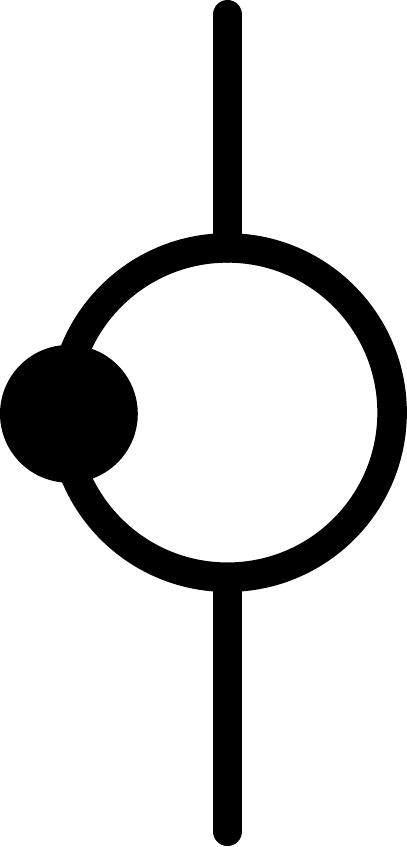}
\end{gathered}
}

\newcommand{\Rdot}{
\begin{gathered}
	\vspace{-3pt}
	\includegraphics[width=.025\textwidth]{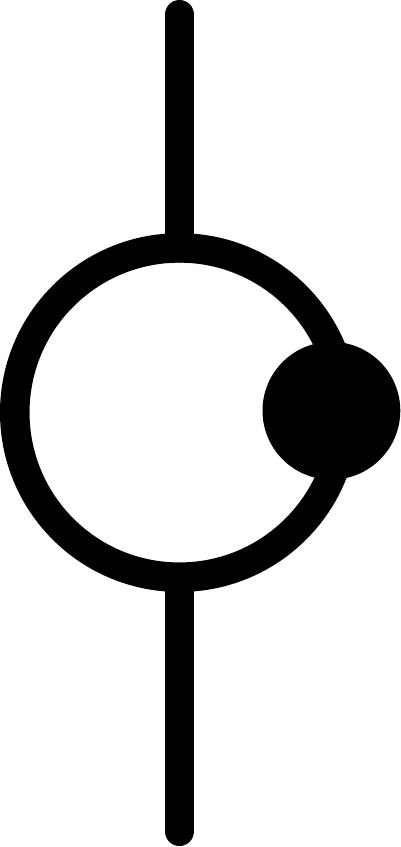}
\end{gathered}
}

\newcommand{\smallTheta}{
\begin{gathered}
	\includegraphics[width=.05\textwidth]{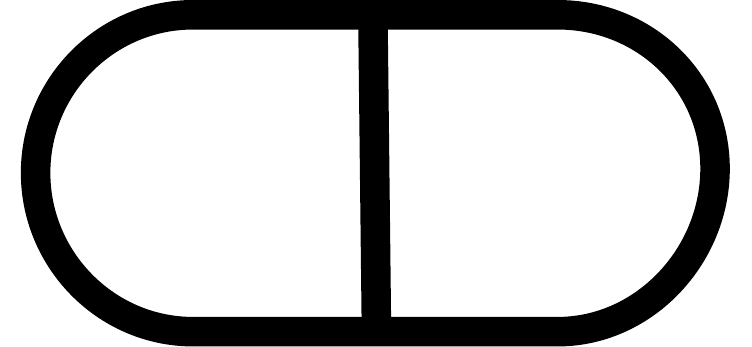}
	\vspace{-3pt}
\end{gathered}
}

\newcommand{\smallThetaR}{
\begin{gathered}
	\includegraphics[width=.05\textwidth]{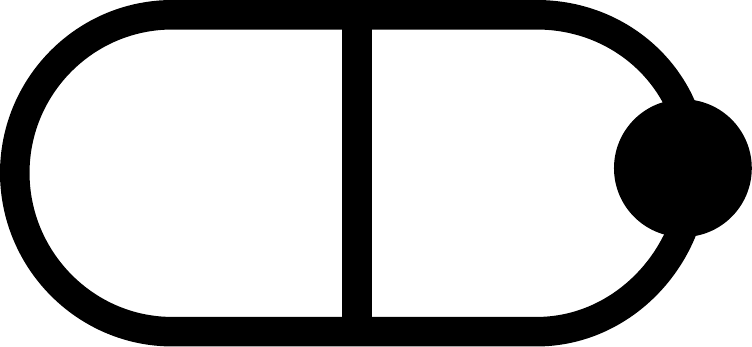}
	\vspace{-3pt}
\end{gathered}
}

\newcommand{\smallThetaL}{
\begin{gathered}
	\includegraphics[width=.05\textwidth]{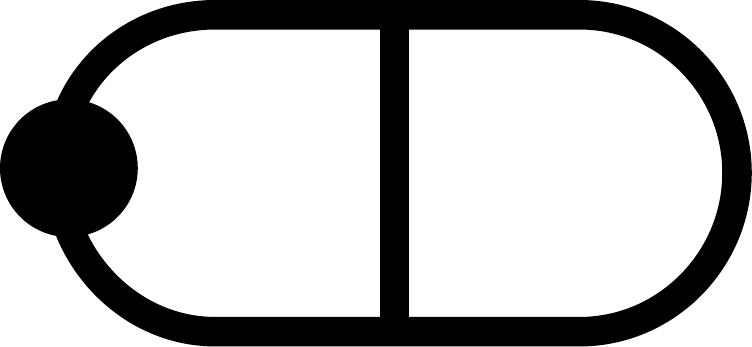}
	\vspace{-3pt}
\end{gathered}
}

\title{The Gysin sequence and the $\sl(N)$ homology of $T(2,m)$}
\author{Joshua Wang}
\date{}
\begin{document}
\maketitle

\begin{abstract}
	The $\sl(N)$ homology of the torus knot or link $T(2,m)$ may be calculated explicitly. By direct comparison, the result is isomorphic to the cohomology of a naturally associated space of $\SU(N)$ representations of the knot group. In honor of Tom Mrowka's 60th birthday, we explain how the Gysin exact sequence may be used to show that these groups are isomorphic without explicitly calculating them. 
\end{abstract}

\section{Introduction}

For $N \ge 2$, the $\sl(N)$ homology $\KR_N(L)$ of an oriented link $L \subset \R^3$ is a bigraded abelian group whose graded Euler characteristic is the $\sl(N)$ link polynomial \cite{MR2391017,MR3545951}. The invariant is the homology of a chain complex defined combinatorially from a link diagram using a cube of resolutions. With the bigrading suppressed, the invariants of the first few torus links $T(2,m)$ are given in the third column of Table~\ref{table:smallTorusLinks}.

\begin{table}[!ht]
	\centering
	\def\arraystretch{1.2}
	\begin{tabular}{ l|l|l|l }
		$m$ & name of $T(2,m)$ & $\KR_N(T(2,m))$ & $\sr R_N(T(2,m))$\\
		\hline
		$1$ & unknot & $\Z^N$ & $\CP^{N-1}$\\
		$2$ & Hopf link & $\Z^{N^2}$ & $\CP^{N-1} \sqcup \F(1,1;N)$\\
		$3$ & trefoil & $\Z^{3N-2}\oplus \Z/N$ & $\CP^{N-1} \sqcup UT\CP^{N-1}$\\
		$4$ & Solomon's knot & $\Z^{N^2 + 2N-2}\oplus \Z/N$ & $\CP^{N-1} \sqcup UT\CP^{N-1} \sqcup \F(1,1;N)$\\
		$5$ & cinquefoil & $\Z^{5N-4} \oplus (\Z/N)^2$ & $\CP^{N-1} \sqcup UT \CP^{N-1} \sqcup UT \CP^{N-1}$
	\end{tabular}
	\captionsetup{width=.8\linewidth}
	\caption{The $\sl(N)$ homology and $\SU(N)$ representation space of the torus link $T(2,m)$ for $m = 1,2,3,4,5$. For all $m$, it turns out that $\KR_N(T(2,m))$ is isomorphic to $H^*(\sr R_N(T(2,m)))$.}
	\label{table:smallTorusLinks}
\end{table}

These groups turn out to arise in another way, entirely unrelated to the definition of $\sl(N)$ homology: they are isomorphic to the ordinary cohomology of spaces naturally associated to these links. Let $\Phi$ be the diagonal matrix $e^{\pi i/N} \mathrm{diag}(-1,1,\ldots,1) \in \SU(N)$, and for any oriented link $L \subset \R^3$, define the $\SU(N)$ representation space $\sr R_N(L)$ to be the space of homomorphisms \[
	\sr R_N(L) \coloneqq \left\{\: \rho\colon \pi_1(\R^3\setminus L) \to \SU(N) \:\big|\: \rho(\mu) \text{ is conjugate to }\Phi \text{ for each meridian }\mu \:\right\}.
\]The space $\sr R_N(T(2,m))$ is straightforward to determine explicitly from a presentation of $\pi_1(\R^3\setminus T(2,m))$ having two meridional generators. The spaces associated to the first few torus links $T(2,m)$ are displayed in the fourth column of Table~\ref{table:smallTorusLinks} where $UT\CP^{N-1}$ is the unit tangent bundle of the projective space $\CP^{N-1}$, and $\F(1,1;N)$ is the partial flag manifold of pairs of orthogonal one-dimensional linear subspaces of $\C^N$. Using the isomorphisms \[
	H^*(\CP^{N-1}) \cong \Z^N \qquad\quad H^*(UT\CP^{N-1}) \cong \Z^{2N-2} \oplus \Z/N \qquad\quad H^*(\F(1,1;N)) \cong \Z^{N^2-N}
\]we see that the $\sl(N)$ homology of $T(2,m)$ matches the cohomology of $\sr R_N(T(2,m))$ for the listed values of $m$. This observation holds for all $m$.

\theoremstyle{plain}
\newtheorem{obs}[thm]{Observation}
\begin{obs}\label{obs:mainObs}
	$\KR_N(T(2,m))$ is isomorphic to $H^*(\sr R_N(T(2,m)))$ as an abelian group for any $m$. 
\end{obs}

In this note, we explain how to prove Observation~\ref{obs:mainObs} without directly observing it. More specifically, we use the Gysin sequence to show that $\KR_N(T(2,m))$ and $H^*(\sr R_N(T(2,m)))$ are isomorphic without needing to explicitly compute either group. Explicitly calculating $H^*(\sr R_N(T(2,m)))$ is simpler than explicitly calculating $\KR_N(T(2,m))$, so Observation~\ref{obs:mainObs} may also be thought of as a way to compute $\KR_N(T(2,m))$. 

\begin{rem}
	Kronheimer and Mrowka's definition of singular instanton homology was motivated by the observation that the Khovanov homology of $T(2,m)$ coincides with the cohomology of $\sr R_2(T(2,m))$ \cite{MR2860345}. Observation~\ref{obs:mainObs} is the natural extension of this observation to $\sl(N)$ homology. Also see \cite{MR3190356,MR3125899}.
\end{rem}

\begin{rem}
	The original definition of $\sl(N)$ homology given by Khovanov and Rozansky \cite{MR2391017} requires $\Q$ coefficients. The $\sl(N)$ homology over $\Q$ of $T(2,m)$ for $m$ odd was first computed by Rasmussen \cite{MR2309174}. The author is unaware of a reference in the literature for a computation of $\KR_N(T(2,m))$ over $\Z$, though the result is known to some experts. 
\end{rem}

\begin{rem}
	A variety of the additional structure on $\sl(N)$ homology is reflected in the $\SU(N)$ representation space. Our proof of Observation~\ref{obs:mainObs} gives the complete bigrading on $\KR_N(T(2,m))$. For more, see \cite{https://doi.org/10.48550/arxiv.2211.08409}. 
\end{rem}

We explain the role of the Gysin sequence in the proof. The key geometric observation is that $UT\CP^{N-1}$ is the total space of a circle bundle $S^1 \to UT\CP^{N-1} \to \F(1,1;N)$ with base space $\F(1,1;N)$. The Gysin sequence of this circle bundle is the long exact sequence \[
	\begin{tikzcd}
		\cdots \ar[r] & H^{i+1}(UT\CP^{N-1}) \ar[r] & H^i(\F(1,1;N)) \ar[r,"\cup\,e"] & H^{i+2}(\F(1,1;N)) \ar[r] & H^{i+2}(UT\CP^{N-1}) \ar[r] & \cdots
	\end{tikzcd}
\]where $e \in H^2(\F(1,1;N))$ is the Euler class of the bundle. This class may be determined explicitly. Because $H^i(\F(1,1;N))$ vanishes when $i$ is odd, the kernel of the map \[
	\begin{tikzcd}[column sep=50pt]
		H^*(\F(1,1;N)) \ar[r,"\cup\,e"] & H^*(\F(1,1;N))
	\end{tikzcd}
\]is isomorphic to $H^{\text{odd}}(UT\CP^{N-1})$ while its cokernel is isomorphic to $H^{\text{even}}(UT\CP^{N-1})$. Phrased differently, $H^*(UT\CP^{N-1})$ is isomorphic to the homology of this two-term chain complex. The $\sl(N)$ complex associated to the trefoil $T(2,3)$ turns out to be homotopy equivalent to the direct sum of this particular two-term complex and the one-term complex associated to the planar unknot diagram. As a result, we are able to deduce that the $\sl(N)$ homology of $T(2,3)$ is isomorphic to $H^*(\CP^{N-1}) \oplus H^*(UT\CP^{N-1}) = H^*(\sr R_N(T(2,3)))$, without having explicitly calculated the homology of the given two-term complex. 

Note that a similar argument using the Gysin sequence of the sphere bundle $S^{2N-3} \to UT\CP^{N-1} \to \CP^{N-1}$ identifies $H^*(UT\CP^{N-1})$ with the homology of the complex \[
	\begin{tikzcd}[column sep=80pt]
		H^*(\CP^{N-1}) \ar[r,"{\cup \,\chi(\CP^{N-1})\cdot[\CP^{N-1}]}"] & H^*(\CP^{N-1})
	\end{tikzcd}
\]where $[\CP^{N-1}] \in H^{2N-2}(\CP^{N-1})$ is the orientation class and $\chi$ is the Euler characteristic. Under the standard identification $H^*(\CP^{N-1}) = \Z[X]/X^N$, the above complex is \[
	\begin{tikzcd}[column sep=75pt]
		\displaystyle\frac{\Z[X]}{X^N} \ar[r,"NX^{N-1}"] & \displaystyle\frac{\Z[X]}{X^N} 
	\end{tikzcd}
\]whose homology is easily seen to be $\Z^{2N-2} \oplus \Z/N$. 

In section~\ref{sec:SU(N)RepSpaces}, we compute the $\SU(N)$ representation spaces of $T(2,m)$ from the definition. We then show that $UT\CP^{N-1}$ is a circle bundle over $\F(1,1;N)$ and compute its Euler class. In section~\ref{sec:sl(N)homology}, we describe the construction of $\sl(N)$ homology, emphasizing connections to the spaces $\CP^{N-1}$ and $\F(1,1;N)$. Finally, we prove Observation~\ref{obs:mainObs} in section~\ref{sec:proofOfMainObs}.

\theoremstyle{definition}
\newtheorem*{ack}{Acknowledgments}
\begin{ack}
	Many thanks to my advisor Peter Kronheimer for his continued guidance, support, and encouragement. This material is based upon work supported by the NSF GRFP through grant DGE-1745303.
\end{ack}

\section{\texorpdfstring{$\SU(N)$}{SU(N)} representation spaces}\label{sec:SU(N)RepSpaces}

We first compute the $\SU(N)$ representation spaces associated to the torus links $T(2,m)$. We then explain how $UT\CP^{N-1}$ is a circle bundle over $\F(1,1;N)$ and compute its Euler class. 

\begin{lem}\label{lem:conjugacyClassOfPhi}
	The conjugacy class of $\Phi \in \SU(N)$ may be identified with $\CP^{N-1}$ by sending each matrix to its eigenspace with eigenvalue $-e^{i\pi/N}$. 
\end{lem}
\begin{proof}
	If $e_1,\ldots,e_N$ is the standard basis of $\C^N$, then $\C\cdot e_1$ is the $(-e^{i\pi/N})$-eigenspace of $\Phi$. Let $U \in \SU(N)$, and note that the $(-e^{i\pi/N})$-eigenspace of $U\Phi U^{-1}$ is $U(\C \cdot e_1)$, which is the image of the line $\C\cdot e_1$ under the linear transformation $U$. Because $\SU(N)$ acts transitively on $\CP^{N-1}$, the map from the conjugacy class of $\Phi$ to $\CP^{N-1}$ described in the lemma is surjective. To see that it is also injective, first note that the $(e^{i\pi/N})$-eigenspace of $U\Phi U^{-1}$ is the image of the span of $e_2,\ldots,e_N$ under the map $U$. This image is precisely the orthogonal complement of $U(\C\cdot e_1)$ because $U$ is unitary. Hence, the $(-e^{i\pi/N})$-eigenspace of $U\Phi U^{-1}$ determines its $(e^{i\pi/N})$-eigenspace and therefore the matrix $U\Phi U^{-1}$ itself. 
\end{proof}

\begin{lem}\label{lem:fundamentalGroupOfT(2,m)}
	The fundamental group of the complement of $T(2,m)$ for $m \ge 1$ is given by the presentation \[
		\pi_1(\R^3\setminus T(2,m)) = \begin{cases}
			\langle\,a,b\:|\:(ab)^ka = (ba)^kb \,\rangle & m \text{ is odd and }m = 2k+1\\
			\langle \, a,b\:|\: (ab)^k = (ba)^k \,\rangle & m \text{ is even and }m = 2k
		\end{cases}
	\]where the generators $a,b$ are meridians.
\end{lem}
\begin{proof}
	This follows from the Wirtinger presentation associated to the diagram of $T(2,m)$ obtained as the braid closure of the two-stranded braid $\sigma_1^m \in B_2 = \Z\langle \sigma_1\rangle$.
\end{proof}

\begin{prop}\label{prop:representationSpaceCalculation}
	The $\SU(N)$ representation space of $T(2,m)$ for $m \ge 1$ is \[
		\sr R_N(T(2,m)) = \begin{cases}
			\CP^{N-1} \sqcup \bigsqcup^k UT\CP^{N-1} & m \text{ is odd and }m=2k+1\\
			\CP^{N-1} \sqcup \bigsqcup^{k-1} UT\CP^{N-1} \sqcup \F(1,1;N) & m\text{ is even and }m=2k
		\end{cases}
	\]
\end{prop}
\begin{proof}
	By Lemma~\ref{lem:fundamentalGroupOfT(2,m)}, a representation $\rho \in \sr R_N(T(2,m))$ is determined by the values $\rho(a)$ and $\rho(b)$, which lie in the conjugacy class of $\Phi$. By Lemma~\ref{lem:conjugacyClassOfPhi}, we may identify the conjugacy class of $\Phi$ with $\CP^{N-1}$. Hence, there is an embedding \[
		\Psi\colon \sr R_N(T(2,m)) \hookrightarrow \CP^{N-1} \x \CP^{N-1}
	\]given by sending $\rho$ to the pair of $(-e^{i\pi/N})$-eigenspaces of $\rho(a)$ and $\rho(b)$. We let $\Lambda_a(\rho)$ and $\Lambda_b(\rho)$ denote these eigenspaces so that $\Psi(\rho) = (\Lambda_a(\rho),\Lambda_b(\rho))$.  
	Next, note that if $U \in \SU(N)$, then the representation $U \rho U^{-1}$ given by $\gamma\mapsto U\rho(\gamma)U^{-1}$ also lies in $\sr R_N(T(2,m))$. This defines a conjugation action of $\SU(N)$ on $\sr R_N(T(2,m))$.
	Now observe that \[
		\Lambda_a(U\rho U^{-1}) = U(\Lambda_a(\rho))
	\]where $U(\Lambda_a(\rho))$ is the image of the line $\Lambda_a(\rho) \subset \C^N$ under the linear transformation $U$. Similarly, we have $\Lambda_b(U\rho U^{-1}) = U(\Lambda_b(\rho))$. Hence, the embedding $\Psi$ is $\SU(N)$-equivariant, where $\SU(N)$ acts on $\CP^{N-1} \x \CP^{N-1}$ diagonally.

	The orbit space of the diagonal $\SU(N)$ action on $\CP^{N-1} \x \CP^{N-1}$ may be identified with the closed interval $[0,\pi/2]$. The identification is given by sending a pair of lines $(\C\cdot v, \C\cdot w)$ to the Hermitian angle $\theta \in [0,\pi/2]$ between them, which is defined by the equation \[
		\|v\|\,\|w\|\cos\theta = |v\cdot w|.
	\]We prove this statement in Lemma~\ref{lem:principalAngle} below. Since $\Psi$ is $\SU(N)$-equivariant, it follows that $\Psi(\sr R_N(T(2,m))) \subset \CP^{N-1}\x\CP^{N-1}$ is a union of orbits. 

	We now show that $\Psi(\sr R_N(T(2,m)))$ is the union of the orbits corresponding to the values of $\theta \in [0,\pi/2]$ satisfying $m\theta \in \Z\pi$. Suppose the orbit corresponding to $\theta \in [0,\pi/2]$ lies in $\Psi(\sr R_N(T(2,m)))$. The lines $\C\cdot e_1$ and $\C\cdot((\cos\theta) e_1 + (\sin\theta) e_2)$ have Hermitian angle $\theta$ so there must be a representation $\rho\colon\pi_1(\R^3\setminus T(2,m)) \to \SU(N)$ for which $\Lambda_a(\rho)$ and $\Lambda_b(\rho)$ are precisely the two given lines. Then $\rho(a) = \Phi$ because the $(-e^{i\pi/N})$-eigenspace of $\Phi$ is $\C\cdot e_1$. Since \[
		U_\theta = \begin{pmatrix}
			\cos\theta & -\sin\theta & 0\\
			\sin\theta & \cos\theta & 0\\
			0 & 0 & \Id_{N-2}
		\end{pmatrix} \in \SU(N)
	\]sends $e_1$ to $(\cos\theta)e_1 + (\sin\theta)e_2$, it follows that $\rho(b) = U_\theta\Phi U_{-\theta}$, using the identity $(U_\theta)^{-1} = U_{-\theta}$. Now note that \[
		\Phi^{-1} = e^{-\pi i/N} \mathrm{diag}(-1,1,\ldots,1) = e^{-2\pi i/N} \Phi
	\]from which it follows that $\rho(a)^{-1} = e^{-2\pi i/N}\rho(a)$ and $\rho(b)^{-1} = e^{-2\pi i/N}\rho(b)$. We may therefore express the relation of Lemma~\ref{lem:fundamentalGroupOfT(2,m)} as \[
		(e^{-2\pi i/N}\rho(b)\rho(a))^m = \Id \qquad\Longleftrightarrow\qquad \begin{cases}
			(\rho(a)\rho(b))^k\rho(a) = (\rho(b)\rho(a))^k\rho(b) & m\text{ is odd and }m = 2k + 1\\
			(\rho(a)\rho(b))^k = (\rho(b)\rho(a))^k & m\text{ is even and }m = 2k
		\end{cases}
	\]Note that $(e^{-i\pi/N}\Phi)U_{-\theta}(e^{-i\pi/N}\Phi) = U_\theta$ because $(e^{-i\pi/N}\Phi) = \mathrm{diag}(-1,1,\ldots,1)$ is reflection in the first coordinate while $U_{-\theta}$ is rotation by $-\theta$ in the first two coordinates. Thus \[
		e^{-2\pi i/N}\rho(b)\rho(a) = e^{-2i\pi/N}U_\theta\Phi U_{-\theta}\Phi = (U_{\theta})^2 = U_{2\theta}
	\]so the relation can be further simplified to $U_{2m\theta} = \Id$. Since $U_{2m\theta}$ is rotation by $2m\theta$ in the first two coordinates, we see that $2m\theta \in 2\Z\pi$ which is equivalent to $m\theta \in \Z\pi$. 

	We have proven that if the orbit corresponding to $\theta$ lies in $\Psi(\sr R_N(T(2,m)))$, then $m\theta \in \Z\pi$. The work we have done also verifies the converse. Given $\theta \in [0,\pi/2]$ satisfying $m\theta\in\Z\pi$, we define the representation $\rho\colon \pi_1(\R^3\setminus T(2,m)) \to \SU(N)$ by sending $\rho(a) = \Phi$ and $\rho(b) = U_\theta\Phi U_{-\theta}$. To finish the proof, we must verify that the orbit in $\CP^{N-1} \x \CP^{N-1}$ corresponding to $\theta \in [0,\pi/2]$ is \[
		\text{Orbit corresponding to }\theta = \begin{cases}
			\CP^{N-1} & \theta = 0\\
			UT\CP^{N-1} & 0 < \theta < \pi/2\\
			\F(1,1;N) & \theta = \pi/2
		\end{cases}
	\]which we prove in Lemma~\ref{lem:orbitCalculation} below. See Figure~\ref{fig:repSpaces}. 
\end{proof}

\begin{figure}[!ht]
	\centering
	\raisebox{-4pt}{\begin{minipage}{.3\textwidth}
			\begin{tikzpicture}
			\node (A) at (0,-.01) {$\bullet$};
			\node (B) at (2,-.01) {$\bullet$};
			\node (C) at (4,-.01) {$\bullet$};
			\node (D) at (0,2) {$\smash{\CP^{N-1}}$};
			\node (E) at (2,2) {$\smash{UT\CP^{N-1}}$};
			\node (F) at (4,2) {$\smash{\F(1,1;N)}$};
			\node (G) at (0,-.4) {$\theta = 0$};
			\node (H) at (2.04,-.4) {$0 < \theta < \smash{\frac{\pi}{2}}$};
			\node (I) at (4.03,-.4) {$\theta = \smash{\frac{\pi}{2}}$};
	
			\draw (0,0) -- (4,0);
			\draw[->]
				(D) edge (A) (E) edge (B) (F) edge (C);
		\end{tikzpicture}
		\end{minipage}}
	\hspace{40pt}
	\begin{minipage}{.5\textwidth}
		\def\arraystretch{1.2}
	\begin{tabular}{ l|l|c }
		$m$ & name of $T(2,m)$ & $\{\,0\leq\theta\leq \pi/2\:|\:m\theta \in \Z\pi\,\}$\\
		\hline
		$1$ & unknot & \raisebox{-3pt}{\begin{tikzpicture}
					\node (A) at (0,-.01) {$\bullet$};
					\node (B) at (4,-.01) {$\phantom{\bullet}$};
					\draw (0,0) -- (4,0);
				\end{tikzpicture}}\\
		$2$ & Hopf link & \raisebox{-3pt}{\begin{tikzpicture}
					\node (A) at (0,-.01) {$\bullet$};
					\node (B) at (4,-.01) {$\bullet$};
					\draw (0,0) -- (4,0);
				\end{tikzpicture}}\\
		$3$ & trefoil & \raisebox{-3pt}{\begin{tikzpicture}
					\node (A) at (0,-.01) {$\bullet$};
					\node (B) at (4,-.01) {$\phantom{\bullet}$};
					\node (C) at (8/3,-.01) {$\bullet$};
					\draw (0,0) -- (4,0);
				\end{tikzpicture}}\\
		$4$ & Solomon's knot & \raisebox{-3pt}{\begin{tikzpicture}
					\node (A) at (0,-.01) {$\bullet$};
					\node (B) at (4,-.01) {$\bullet$};
					\node (C) at (2,-.01) {$\bullet$};
					\draw (0,0) -- (4,0);
				\end{tikzpicture}}\\
		$5$ & cinquefoil & \raisebox{-3pt}{\begin{tikzpicture}
					\node (A) at (0,-.01) {$\bullet$};
					\node (B) at (4,-.01) {$\phantom{\bullet}$};
					\node (C) at (8/5,-.01) {$\bullet$};
					\node (D) at (16/5,-.01) {$\bullet$};
					\draw (0,0) -- (4,0);
				\end{tikzpicture}}
	\end{tabular}
	\end{minipage}
	\captionsetup{width=.8\linewidth}
	\caption{On the left are the orbits of the diagonal $\SU(N)$ action on $\CP^{N-1} \x \CP^{N-1}$. On the right are the orbits whose union is $\sr R_N(T(2,m))$ for $m = 1,2,3,4,5$.}
	\label{fig:repSpaces}
\end{figure}
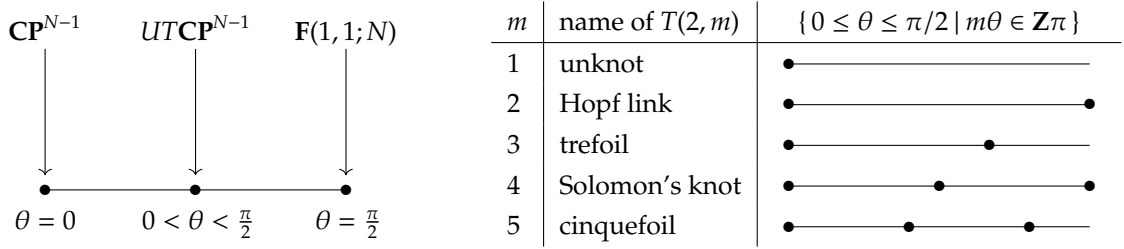

\begin{lem}\label{lem:principalAngle}
	The map from $\CP^{N-1} \x \CP^{N-1}$ to $[0,\pi/2]$ sending a pair of lines to the Hermitian angle between them gives an identification between the orbit space of the diagonal $\SU(N)$ action on $\CP^{N-1} \x \CP^{N-1}$ and $[0,\pi/2]$. 
\end{lem}
\begin{proof}
	The Hermitian angle is invariant under the diagonal action because \[
		\frac{|Uv \cdot Uw|}{\|Uv\|\,\|Uw\|} = \frac{|v\cdot w|}{\|v\|\,\|w\|}
	\]for any $U \in \SU(N)$. It suffices to show that $\SU(N)$ acts transitively on pairs of lines having a fixed Hermitian angle between them. If $e_1,\ldots,e_N$ is the standard basis for $\C^N$, then any pair of lines may be mapped into the span of $e_1$ and $e_2$ by an element of $\SU(N)$. This follows from the fact that $\SU(N)$ acts transitively on the Grassmannian $\G(2,N)$ of $2$-planes in $\C^N$. Using the inclusion $\SU(2) \x \Id_{N-2} \subset \SU(N)$ by block diagonal matrices, we have reduced to the case that $N = 2$. 

	First, we describe the geometric picture. Given two points in $\CP^1$, the Hermitian angle between them is precisely the distance between them in the Fubini--Study metric, which is the round metric on $S^2$ with radius 1/2. Furthermore, $\SU(2)$ acts on $\CP^1$ through isometries, realizing $\SU(2)$ as the double cover of $\SO(3)$. The quotient of $S^2 \x S^2$ by the diagonal $\SO(3)$ action is clearly a closed interval given by the distance between the points, as claimed. 

	With the picture in place, it is easy to verify the result in coordinates and formulas. We think of $[0:1]$ as a basepoint, and consider the set of points with a fixed Hermitian angle $\theta$ from $[0:1]$. If $\theta = 0$, then the set is the singleton $[0:1]$, while if $\theta = \pi/2$, then the set is the singleton $[1:0]$. Any other point may be described in homogeneous coordinates as $[\zeta:1]$ with $\zeta \neq 0$. The Hermitian angle $\theta$ between $[\zeta:1]$ and $[0:1]$ is given by \[
		\cos\theta = \frac{1}{\sqrt{1 + |\zeta|^2}}
	\]so the points with a fixed Hermitian angle $0 < \theta < \pi/2$ from $[0,1]$ form a circle. They are the points $[\zeta:1]$ with a fixed value of $|\zeta| > 0$. It suffices to show that the matrices in $\SU(2)$ that fix $[0:1]$ act transitively on each of these circles. The matrices fixing $[0:1]$ are precisely the diagonal matrices $\mathrm{diag}(\omega,\omega^{-1})$ with $\omega \in \U(1)$. The matrix $\mathrm{diag}(\omega,\omega^{-1})$ sends $[\zeta:1]$ to $[\omega^2\zeta:1]$, which is transitive as required. 
\end{proof}

\begin{lem}\label{lem:orbitCalculation}
	The orbit of the diagonal $\SU(N)$ action on $\CP^{N-1} \x \CP^{N-1}$ consisting of pairs of lines with a fixed Hermitian angle $\theta \in [0,\pi/2]$ between them is given by \[
		\text{Orbit corresponding to }\theta = \begin{cases}
			\CP^{N-1} & \theta = 0\\
			UT\CP^{N-1} & 0 < \theta < \pi/2\\
			\F(1,1;N) & \theta = \pi/2
		\end{cases}
	\]
\end{lem}
\begin{proof}
	If the Hermitian angle between lines $\Lambda_A$ and $\Lambda_B$ is zero, then $\Lambda_A = \Lambda_B$, and the orbit is clearly a copy of $\CP^{N-1}$. The Hermitian angle between $\Lambda_A$ and $\Lambda_B$ is $\pi/2$ if and only if they are orthogonal, so their orbit is $\F(1,1;N)$ by definition of $\F(1,1;N)$. 

	To describe the orbits corresponding to $0 < \theta < \pi/2$, we again first describe the geometric picture. Given two points in $\CP^{N-1}$, the Hermitian angle between them is precisely the distance between them in the Fubini--Study metric. The exponential map at a point $\Lambda_A \in \CP^{N-1}$ is a diffeomorphism when restricted to the open ball of radius $\pi/2$, and it is surjective on the closed ball of radius $\pi/2$. The image of the sphere of radius $0 <\theta\leq\pi/2$ under the exponential is precisely the set of points $\Lambda_B \in \CP^{N-1}$ whose Hermitian angle from $\Lambda_A$ is $\theta$. Thus the exponential map, precomposed with rescaling, allows us to identity $UT\CP^{N-1}$ with the orbit corresponding $0 < \theta < \pi/2$.

	We verify the picture in coordinates and formulas. Think of $[0:\cdots:0:1]$ as a basepoint, and note that the points $[z_1:\cdots:z_N]$ with $z_N = 0$ all have Hermitian angle $\pi/2$ from $[0:\cdots:0:1]$. The Hermitian angle $\theta$ between $[0:\cdots:0:1]$ and $[\zeta_1:\cdots:\zeta_{N-1}:1]$ is given by \[
		\cos\theta = \frac{1}{\sqrt{1 + |\zeta_1|^2 + \cdots + |\zeta_{N-1}|^2}}
	\]so the points with a fixed Hermitian angle $0 < \theta < \pi/2$ from $[0:\cdots:0:1]$ form a $(2N-3)$-sphere consisting of those points $[\zeta_1:\cdots:\zeta_{N-1}:1]$ with a fixed value of $|\zeta_1|^2 + \cdots + |\zeta_{N-1}|^2$. Thus the orbit corresponding to $\theta$ is a $(2N-3)$-sphere bundle over $\CP^{N-1}$, and it clear from our description that it is the unit tangent bundle. 
\end{proof}

We now explain how $UT\CP^{N-1}$ is a circle bundle over $\F(1,1;N)$, and compute its Euler class. First, note that there are two tautological line bundles $\sr A,\sr B$ over $\F(1,1;N)$. The fiber of $\sr A$ over the point $(\Lambda_A,\Lambda_B) \in \F(1,1;N)$ is the $1$-dimensional vector space $\Lambda_A$, while the fiber of $\sr B$ over $(\Lambda_A,\Lambda_B)$ is $\Lambda_B$. 

\begin{prop}\label{prop:eulerClass}
	The space $UT\CP^{N-1}$ is a circle bundle over $\F(1,1;N)$ whose Euler class $e \in H^2(\F(1,1;N))$ is \[
		e = c_1(\sr A) - c_1(\sr B).
	\]
\end{prop}
\begin{proof}
	View $UT\CP^{N-1}$ as the space of pairs $(\Lambda,v)$ where $\Lambda \in \CP^{N-1}$ and $v$ is a tangent vector at $\Lambda$ of length $\pi/2$ in the Fubini--Study metric. Then the exponential $\exp_\Lambda(v)$ is a point in $\CP^{N-1}$ whose Hermitian angle from $\Lambda$ is $\pi/2$. The bundle map $UT\CP^{N-1} \to \F(1,1;N)$ is given by $(\Lambda,v) \mapsto (\Lambda,\exp_\Lambda(v))$. Both $UT\CP^{N-1}$ and $\F(1,1;N)$ are fiber bundles over $\CP^{N-1}$, with fibers $S^{2N-3}$ and $\CP^{N-2}$, respectively. The map $UT\CP^{N-1} \to \F(1,1;N)$ sends fibers to fibers, and the map $S^{2N-3} \to \CP^{N-2}$ is just the standard quotient map with circle fiber generalizing the Hopf fibration. Alternatively, we may view $UT\CP^{N-1}$ as the orbit in $\CP^{N-1} \x \CP^{N-1}$ corresponding to a fixed Hermitian angle $0 < \theta < \pi/2$. Given a point $(\Lambda_A,\Lambda_B)$ in this orbit, consider the unique projective line $\ell$ in $\CP^{N-1}$ through these points. The map $UT\CP^{N-1} \to \F(1,1;N)$ can also be defined by sending $(\Lambda_A,\Lambda_B)$ to $(\Lambda_A,\Lambda_C)$ where $\Lambda_C$ is the antipodal point of $\Lambda_A$ in $\ell$. 

	We compute the Euler class of this circle bundle. First, consider the map $\pi\colon\F(1,1;N) \to \G(2,N)$ sending a pair of orthogonal lines to their $2$-dimensional span in $\C^N$. This map is a fiber bundle with fiber $\CP^1$. Let $V$ be the vertical tangent bundle of $\F(1,1;N)$ with respect to $\pi$ so that \[
		T\F(1,1;N) = V \oplus \pi^*(T\G(2,N)).
	\]The bundle $V$ is a complex line bundle over $\F(1,1;N)$ whose unit circle bundle is $UT\CP^{N-1}$. To see this, fix $0 < \theta < \pi/2$ and consider the space of quadruples $(\Lambda_A,\Lambda_B,\Lambda_C,\Lambda_D)$ of points in $\CP^{N-1}$ where \begin{itemize}[noitemsep]
		\item the Hermitian angle between $\Lambda_A$ and $\Lambda_B$ is $\theta$,
		\item $\Lambda_C$ is the antipodal point of $\Lambda_A$ on the unique projective line $\ell$ through $\Lambda_A$ and $\Lambda_B$,
		\item $\Lambda_D$ is the antipodal point of $\Lambda_B$ on $\ell$.
	\end{itemize}Clearly, the pair $(\Lambda_A,\Lambda_B)$ determines the quadruple so this space can be identified with $UT\CP^{N-1}$. On the other hand, we can view this quadruple as a pair $((\Lambda_A,\Lambda_C),(\Lambda_B,\Lambda_D))$ of points in $\F(1,1;N)$. The point $(\Lambda_B,\Lambda_D)$ can be identified with the exponential of a vector in the fiber of $V$ over $(\Lambda_A,\Lambda_C)$ of length $\theta$ under a natural generalization of the Fubini--Study metric to $\F(1,1;N)$. 

	For cosmetic reasons, we orient the circle bundle $UT\CP^{N-1}$ to be opposite that of the natural orientation on the unit circle bundle of $V$. The Euler class of $UT\CP^{N-1}$ is therefore the negative of the first Chern class of $V$, so \[
		-e = c_1(V) = c_1(T\F(1,1;N)) - \pi^*(c_1(T\G(2,N))).
	\]Let $\sr C \to \F(1,1;N)$ be the rank $N-2$ vector bundle whose fiber over $(\Lambda_A,\Lambda_B)$ is the orthogonal complement of $\Lambda_A\oplus\Lambda_B$, and note that $\sr A \oplus \sr B \oplus \sr C$ is the trivial rank $N$ bundle. It follows that $c_1(\sr C) = - c_1(\sr A) - c_1(\sr B)$. Let $S,Q$ be the tautological vector bundles over $\G(2,N)$ so that their fibers over the point $\Lambda \in \G(2,N)$ are $\Lambda$ and $\Lambda^\perp$ respectively. Note that $\pi^*(S) = \sr A \oplus \sr B$ and $\pi^*(Q) = \sr C$. Recalling that $T(\G(2,N)) \cong \Hom(S,Q) \cong S^* \otimes Q$ and $c_1(V\otimes W) = \rk(V)c_1(W) + \rk(W)c_1(V)$, we have \begin{align*}
		\pi^*(c_1(T\G(2,N))) &= \pi^*(-\rk(Q)c_1(S) + \rk(S)c_1(Q))\\
		&= -(N-2)(c_1(\sr A) + c_1(\sr B)) + 2(-c_1(\sr A) - c_1(\sr B))\\
		&= -N(c_1(\sr A) + c_1(\sr B)).
	\end{align*}Since $T\F(1,1;N) \cong \Hom(\sr A,\sr B) \oplus \Hom(\sr A,\sr C) \oplus \Hom(\sr B,\sr C)$ \cite{MR431194}, we find that \begin{align*}
		c_1(T\F(1,1;N)) &= (-c_1(\sr A) + c_1(\sr B)) + (-(N-2)c_1(\sr A) - c_1(\sr A) - c_1(\sr B)) + (-(N-2)c_1(\sr B) - c_1(\sr A) - c_1(\sr B))\\
		&= (-N-1) c_1(\sr A) + (- N + 1)c_1(\sr B).
	\end{align*}Thus, we have \[
		-e = (-N-1)c_1(\sr A) + (-N+1)c_1(\sr B) + N (c_1(\sr A) + c_1(\sr B)) = c_1(\sr B) - c_1(\sr A)
	\]as claimed.
\end{proof}

\section{\texorpdfstring{$\sl(N)$}{sl(N)} homology}\label{sec:sl(N)homology}

We describe the construction of $\sl(N)$ homology, drawing from the exposition in \cite[section 2]{wang2021sln}. 

\begin{df}
	A \textit{web $W$ (with labels $1$ and $2$)} is a trivalent graph embedded in the plane where multiple edges and circles without vertices are permitted. Each edge is oriented and labeled $1$ or $2$, and the orientations and labels of the three edges incident to each vertex must match one of the two models given in Figure~\ref{fig:mergeAndSplitVertices}.
\end{df}

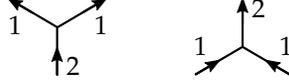
\begin{figure}[!ht]
	\centering
	\labellist
	\pinlabel $2$ at 130 20
	\pinlabel $1$ at 15 100
	\pinlabel $1$ at 185 100
	\endlabellist
	\includegraphics[width=.08\textwidth]{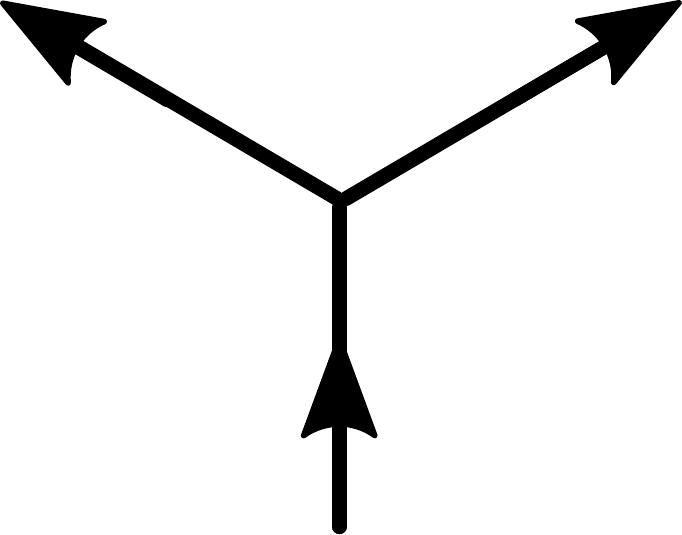}\hspace{30pt}
	\labellist
	\pinlabel $2$ at 130 135
	\pinlabel $1$ at 15 60
	\pinlabel $1$ at 185 60
	\endlabellist
	\includegraphics[width=.08\textwidth]{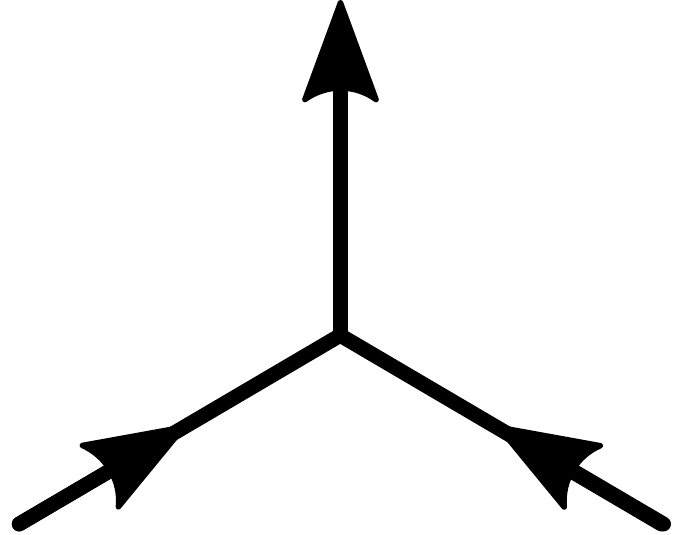}
	\captionsetup{width=.8\linewidth}
	\caption{Orientations and labels of the edges incident to a vertex in a web.}
	\label{fig:mergeAndSplitVertices}
\end{figure}

Associated to a web $W$ and integer $N \ge 2$ is a polynomial $P_N(W) \in \Z[q,q^{-1}]$ with nonnegative coefficients called its MOY polynomial \cite{MR1659228}. This polynomial is sometimes denoted $\langle W\rangle$, though it does depend on $N$. There are local relations that the polynomial satisfies which also determine it, which we summarize in the following result (see for example \cite[Figure 3]{MR2491657}). The MOY polynomial is multiplicative under disjoint union, so in particular, the MOY polynomial of the empty web is $1$. 

\theoremstyle{plain}
\newtheorem*{MOYCalc}{MOY Calculus}

\begin{MOYCalc}
	The MOY polynomial of a web is determined by the local relations in Figure~\ref{fig:MOYcalculus}. 
\end{MOYCalc}

Figure~\ref{fig:MOYcalculus} uses the following notation. For integers $k,N \ge 0$, the polynomials $[N],\qbinom{N}{k} \in \Z[q,q^{-1}]$ are \[
	[N] = \frac{q^N-q^{-N}}{q-q^{-1}} = q^{N-1} + q^{N-3} + \cdots + q^{-N+1} \qquad\qquad \qbinom{N}{k} = \begin{cases}
		\displaystyle\frac{[N]!}{[k]![N-k]!} & 0 \leq k \leq N\\
		0 & \text{else}
	\end{cases}
\]where $[k]! = [k]\cdots[2][1]$. These polynomials are just the Poincar\'e polynomials of $\CP^{k-1}$ and the complex Grassmannian $\G(k,n)$, shifted downwards to be invariant under $q\mapsto q^{-1}$. After a similar shift, the Poincar\'e polynomial of $\F(1,1;N)$ is $[N][N-1]$. This can be seen from the fact that the Serre spectral sequence of the bundle $\CP^{N-2} \to \F(1,1;N) \to \CP^{N-1}$ collapses on the second page for lacunary reasons. 

\begin{figure}[!ht]
	\centering
	\[
		\Big\langle\,\begin{gathered}
			\labellist
			\pinlabel $1$ at 110 100
			\endlabellist
			\includegraphics[width=.04\textwidth]{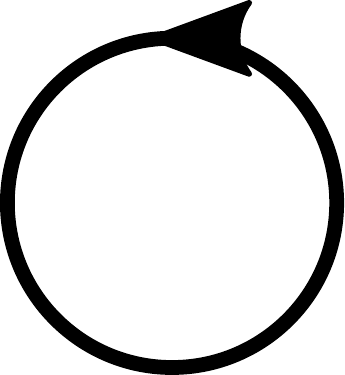}
		\end{gathered}\:\:\:\Big\rangle = \Big\langle\,\begin{gathered}
			\labellist
			\pinlabel $1$ at 110 100
			\endlabellist
			\includegraphics[width=.04\textwidth]{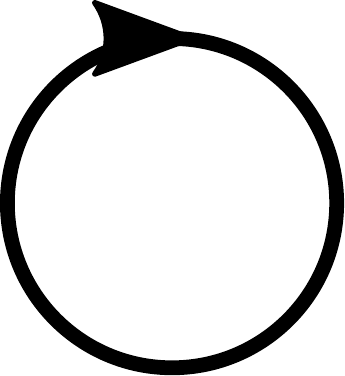}
		\end{gathered}\:\:\:\Big\rangle = [N] \:\Big\langle \qquad\Big\rangle \qquad\qquad \Big\langle\,\begin{gathered}
			\labellist
			\pinlabel $2$ at 110 100
			\endlabellist
			\includegraphics[width=.04\textwidth]{counterclockwiseCircle}
		\end{gathered}\:\:\:\Big\rangle = \Big\langle\,\begin{gathered}
			\labellist
			\pinlabel $2$ at 110 100
			\endlabellist
			\includegraphics[width=.04\textwidth]{clockwiseCircle}
		\end{gathered}\:\:\:\Big\rangle = \qbinom{N}{2}\:\Big\langle \qquad\Big\rangle
	\]\[
		\Bigg\langle \:\begin{gathered}
			\labellist
			\pinlabel $2$ at 90 20
			\pinlabel $2$ at 90 220
			\endlabellist
			\includegraphics[width=.045\textwidth]{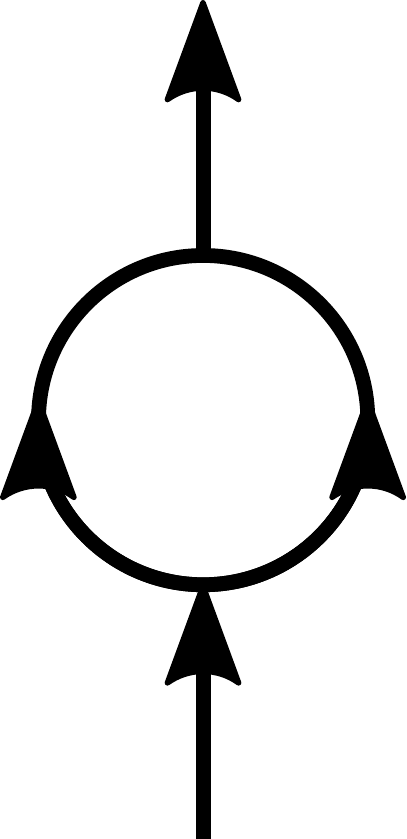}
		\end{gathered} \:\:\:\Bigg\rangle = [2]\: \Bigg\langle\:\: \begin{gathered}
			\labellist
			\pinlabel $2$ at 40 20
			\endlabellist
			\includegraphics[width=.0085\textwidth]{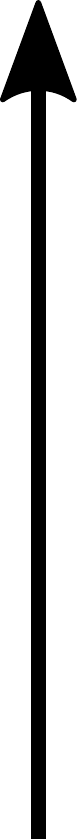}
		\end{gathered} \:\:\:\Bigg\rangle \qquad \qquad\qquad \Bigg\langle \:\begin{gathered}
			\labellist
			\pinlabel $2$ at 140 120
			\endlabellist
			\includegraphics[width=.045\textwidth]{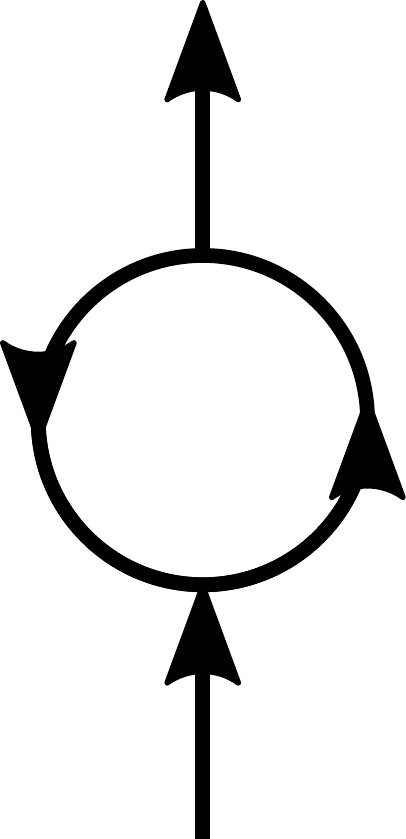}
		\end{gathered} \:\:\:\Bigg\rangle = [N-1]\: \Bigg\langle\:\: \begin{gathered}
			\includegraphics[width=.0085\textwidth]{verticalMOY}
		\end{gathered} \:\:\:\Bigg\rangle = \Bigg\langle \:\:\:\begin{gathered}
			\labellist
			\pinlabel $2$ at -20 120
			\endlabellist
			\includegraphics[width=.045\textwidth]{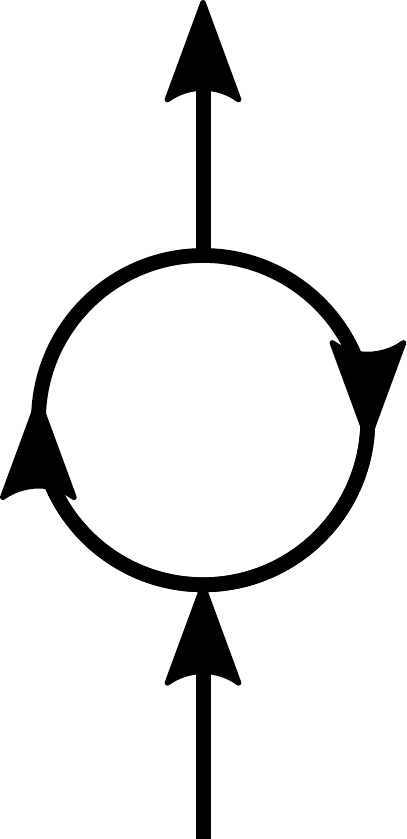}
		\end{gathered} \:\Bigg\rangle
	\]\[
		\Bigg\langle\: \begin{gathered}
			\labellist
			\pinlabel $2$ at 20 100
			\pinlabel $2$ at 185 100
			\endlabellist
			\includegraphics[width=.08\textwidth]{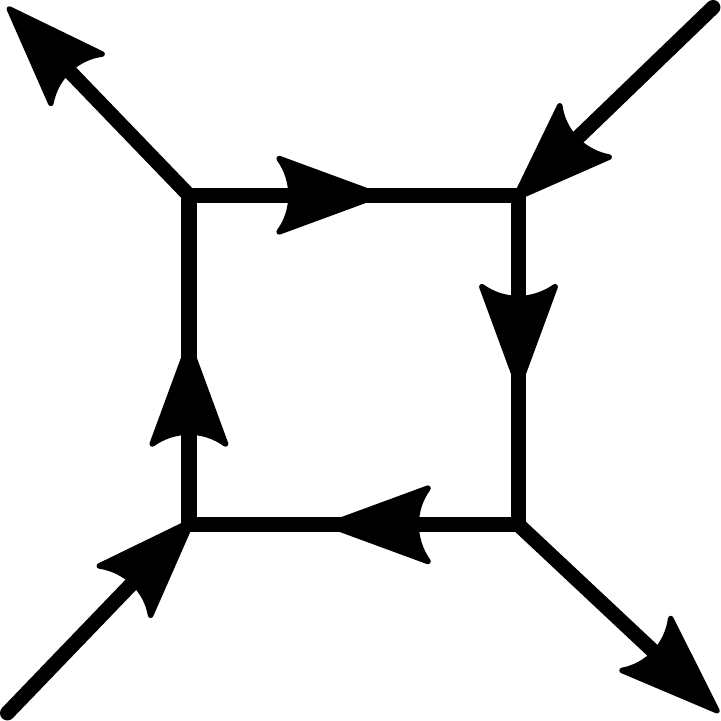}
		\end{gathered} \:\:\Bigg\rangle = \Bigg\langle\: \begin{gathered}
			\includegraphics[width=.08\textwidth]{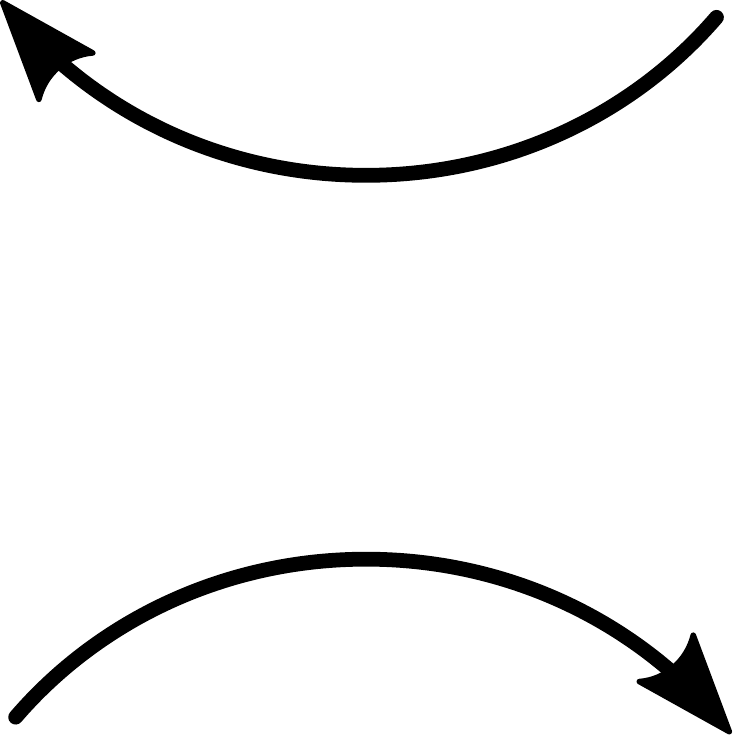}
		\end{gathered} \:\:\Bigg\rangle + [N-2] \:\Bigg\langle \: \begin{gathered}
			\includegraphics[width=.08\textwidth]{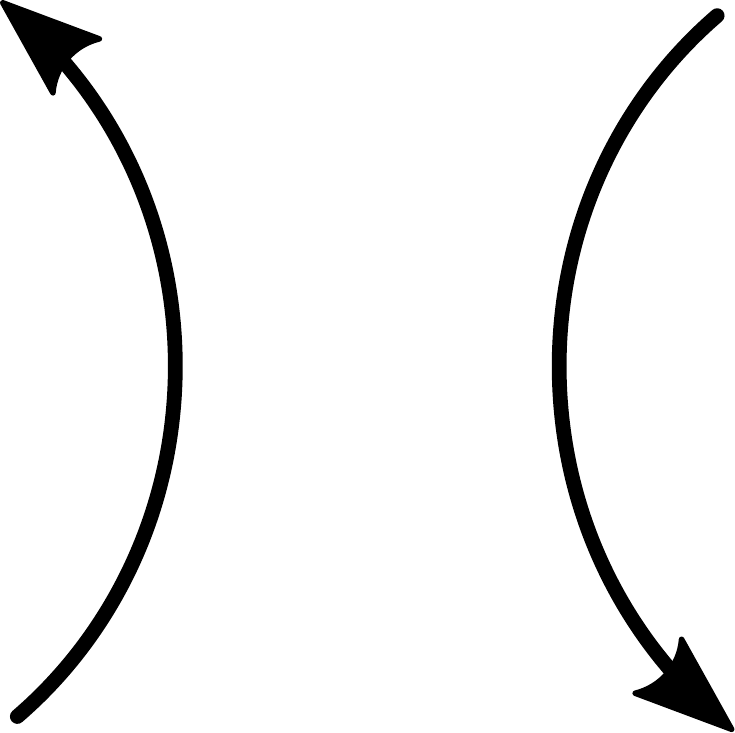}
		\end{gathered}\:\:\Bigg\rangle
	\]\[
		\Bigg\langle\:\:\begin{gathered}
			\labellist
			\pinlabel $2$ at 110 140
			\pinlabel $2$ at 110 260
			\pinlabel $2$ at 300 190
			\endlabellist
			\includegraphics[width=.12\textwidth]{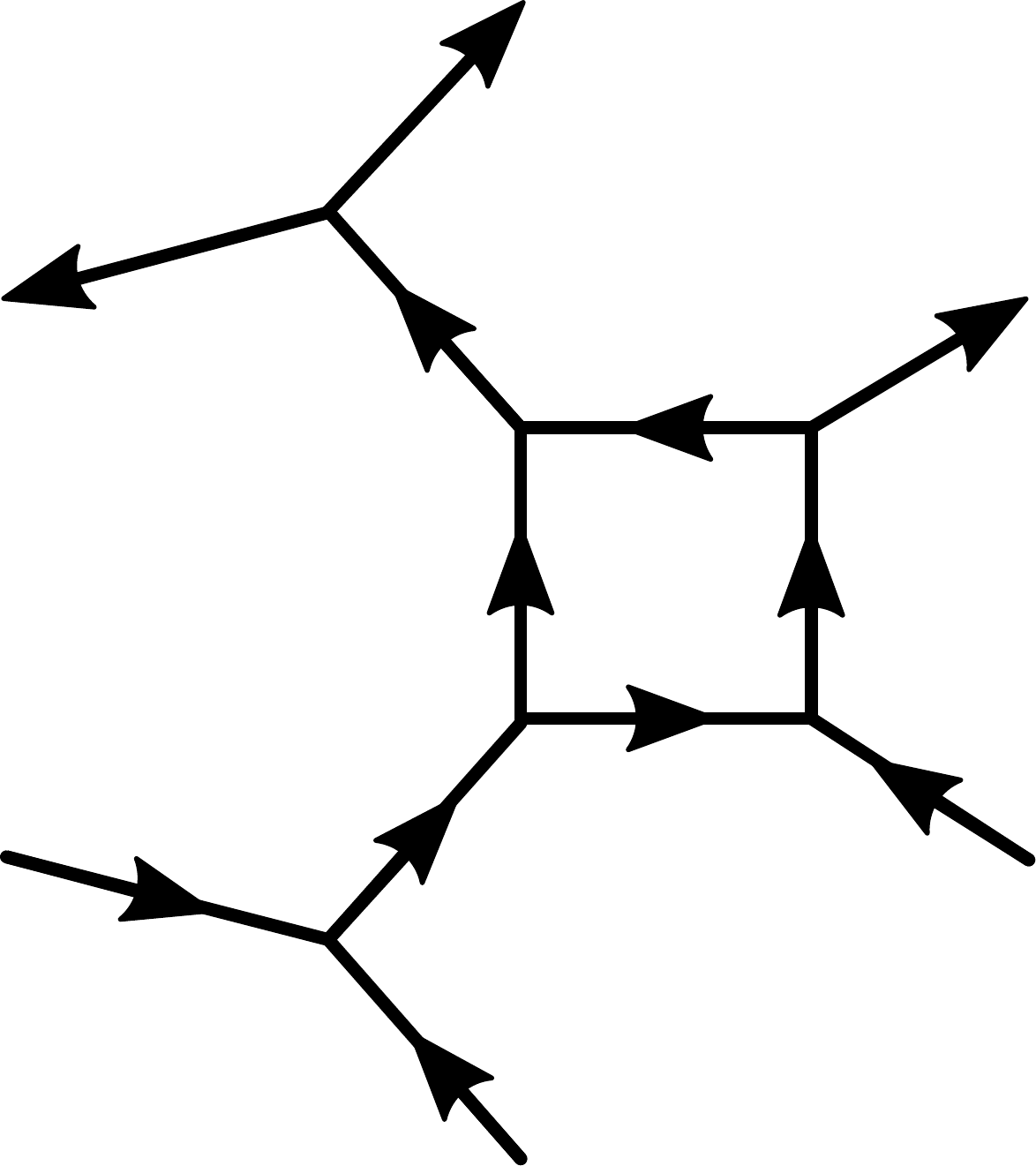}
		\end{gathered}\:\:\:\Bigg\rangle + \Bigg\langle\:\:\begin{gathered}
			\labellist
			\pinlabel $2$ at 300 190
			\endlabellist
			\includegraphics[width=.12\textwidth]{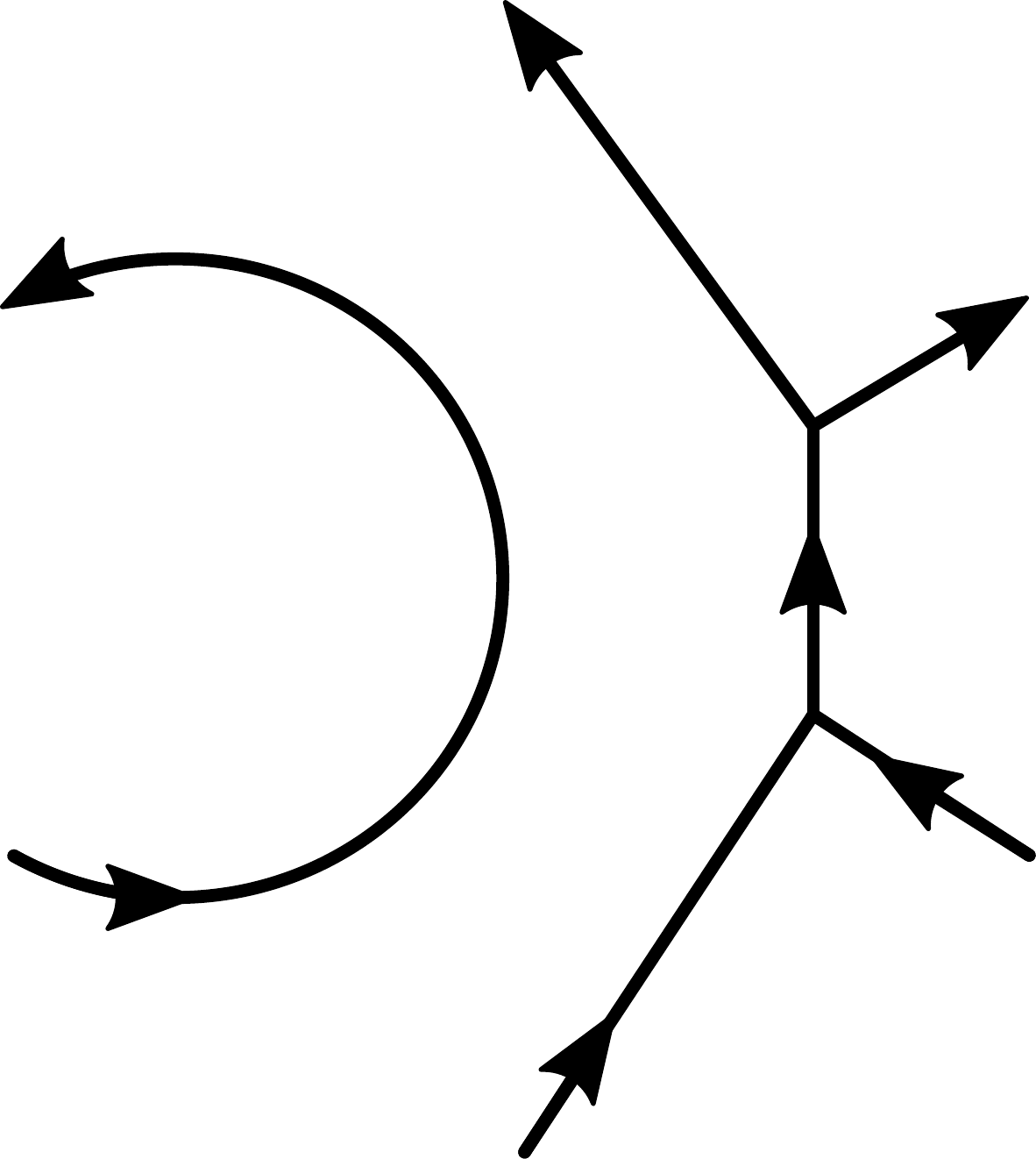}
		\end{gathered}\:\:\:\Bigg\rangle = \Bigg\langle\:\: \begin{gathered}
			\labellist
			\pinlabel $2$ at 40 190
			\pinlabel $2$ at 232 140
			\pinlabel $2$ at 232 260
			\endlabellist
			\includegraphics[width=.12\textwidth]{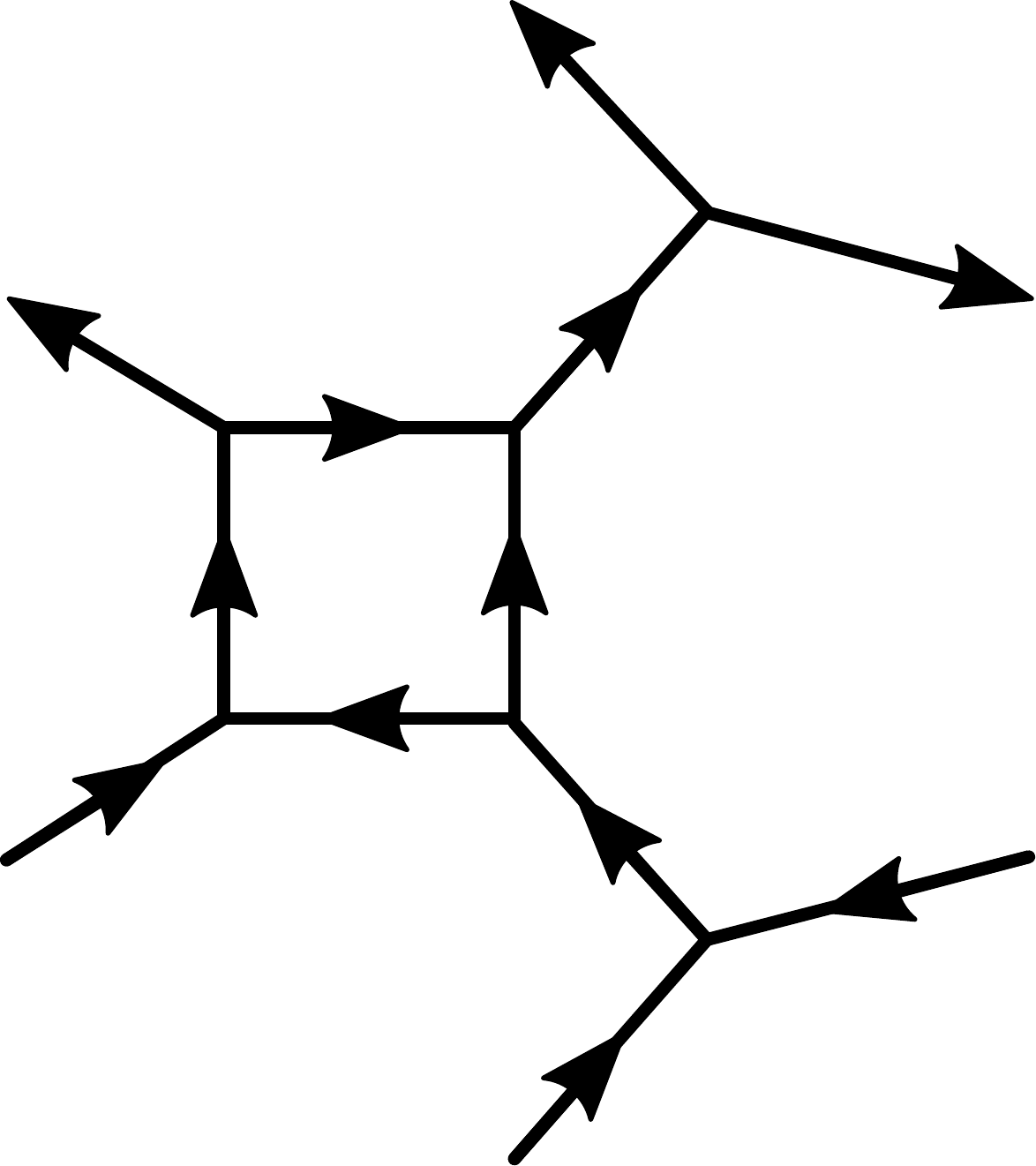}
		\end{gathered}\:\:\:\Bigg\rangle + \Bigg\langle\:\:\begin{gathered}
			\labellist
			\pinlabel $2$ at 40 190
			\endlabellist
			\includegraphics[width=.12\textwidth]{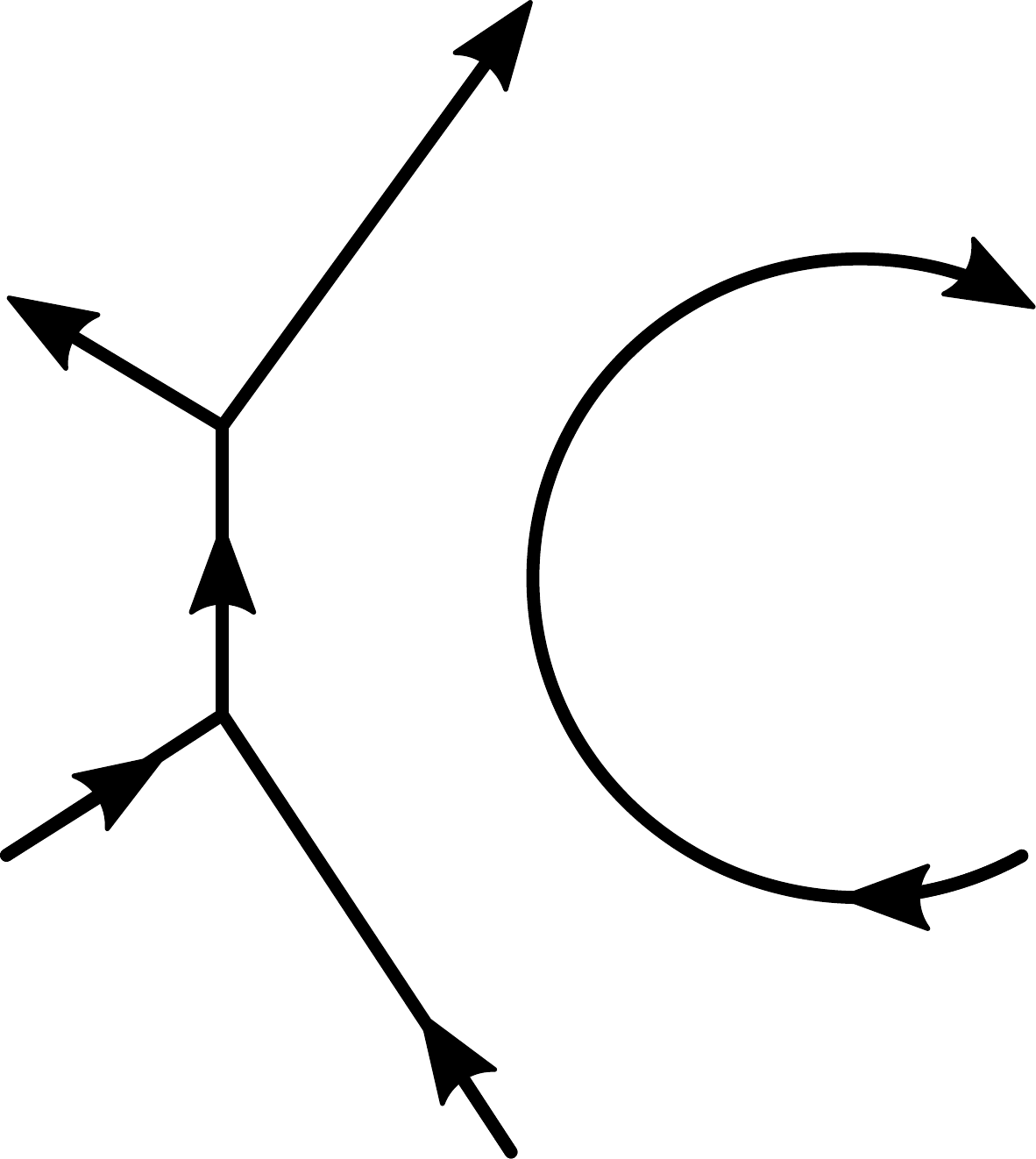}
		\end{gathered}\:\:\:\Bigg\rangle
	\]
	\captionsetup{width=.8\linewidth}
	\caption{Local relations in MOY calculus. Edges without explicit labels are labeled $1$.}
	\label{fig:MOYcalculus}
\end{figure} 

Given a link diagram, the $\sl(N)$ polynomial can be computed as a weighted sum of MOY polynomials obtained by taking complete resolutions of the diagram, where one of the two possible resolutions introduces trivalent vertices and an edge labeled $2$. See Figure~\ref{fig:slNpolynomialFromMOY}.

\begin{figure}[!ht]
	\centering \begin{align*}
		\begin{gathered}
		\includegraphics[width=.08\textwidth]{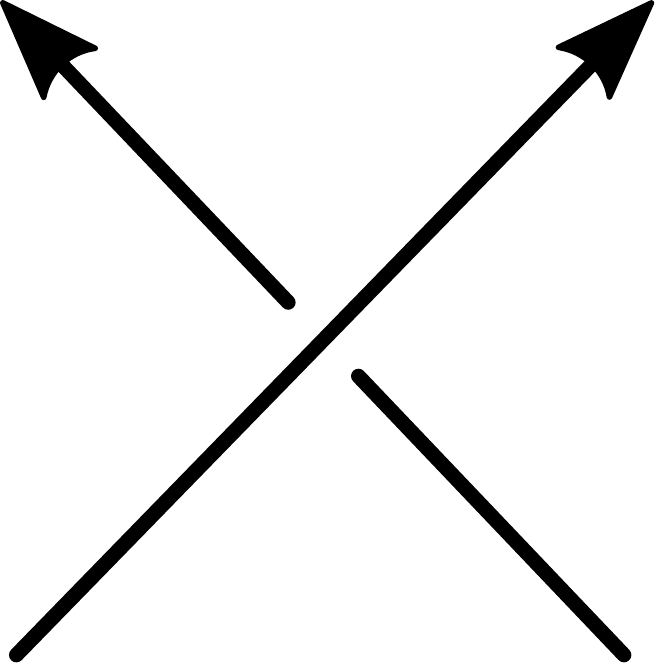}
		\end{gathered} &= q^{N-1}\begin{gathered}
		\includegraphics[width=.08\textwidth]{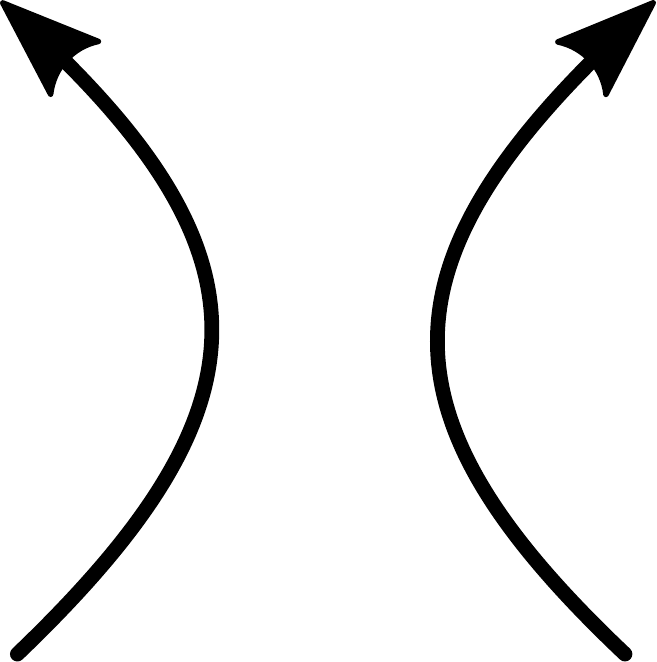}
		\end{gathered} - q^{+N} \begin{gathered}
			\labellist
			\pinlabel $2$ at 130 90
			\endlabellist
			\includegraphics[width=.08\textwidth]{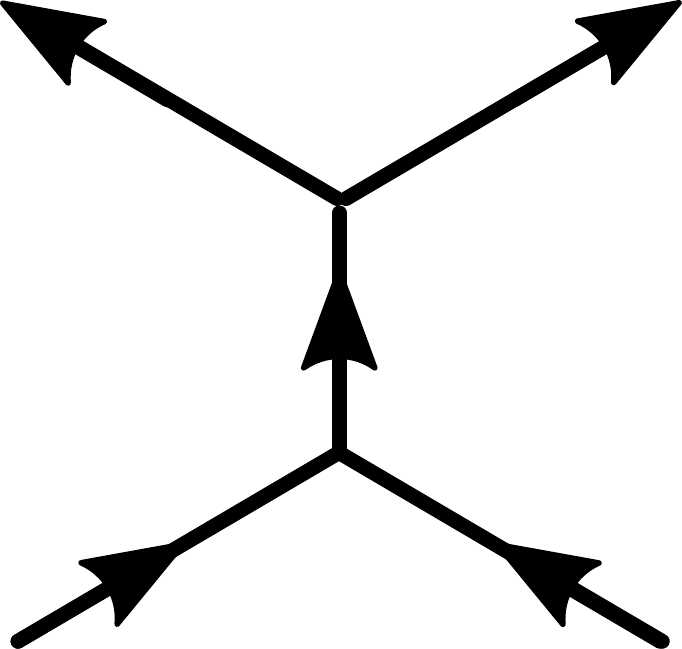}
		\end{gathered}\\
\begin{gathered}
		\includegraphics[width=.08\textwidth]{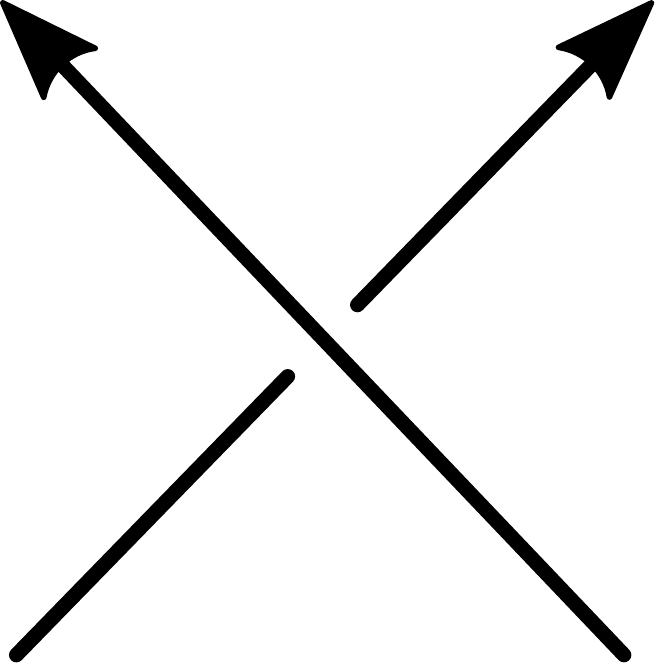}
	\end{gathered} &= q^{1 - N}\begin{gathered}
		\includegraphics[width=.08\textwidth]{orientedResolution.pdf}
	\end{gathered} - q^{-N} \begin{gathered}
		\labellist
			\pinlabel $2$ at 130 90
			\endlabellist
		\includegraphics[width=.08\textwidth]{otherResolution}
	\end{gathered}
	\end{align*}
	\vspace{-10pt}
	\captionsetup{width=.8\linewidth}
	\caption{The $\sl(N)$ polynomial from the MOY polynomial. Edges without explicit labels are labeled $1$.}
	\label{fig:slNpolynomialFromMOY}
\end{figure}

\begin{example}\label{example:thetaWebPolynomial}
	Consider applying the first relation of Figure~\ref{fig:slNpolynomialFromMOY} to the positive crossing in the diagram of unknot described by $T(2,1)$. The first resolution consists of two planar loops which can be thought of as $T(2,0) = O \sqcup O$, while the second resolution is called the \textit{$\theta$-web}: \[
		\begin{gathered}
			\labellist
			\pinlabel $2$ at 135 55
			\endlabellist
			\includegraphics[width = .11\textwidth]{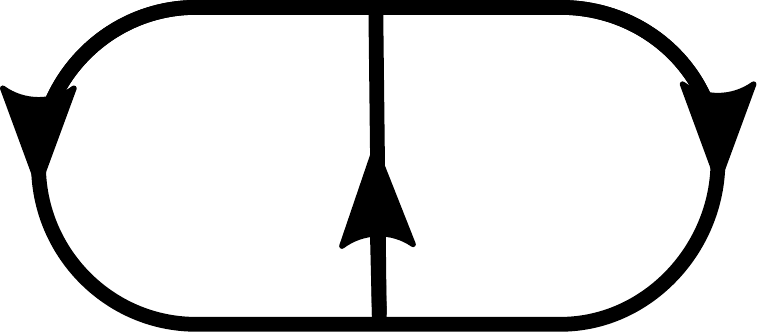}
		\end{gathered}
	\]The relation states that $P_N(T(2,1)) = q^{N-1} P_N(O\sqcup O) - q^N P_N(\smallTheta)$. Using MOY calculus (Figure~\ref{fig:MOYcalculus}), we find that $P_N(O \sqcup O) = [N]^2$ and $P_N(\smallTheta) = [N][N-1]$ from which it follows that \begin{align*}
		P_N(T(2,1)) &= [N]( q^{N-1}[N] - q^N [N-1])\\
		&= [N] \left(\frac{q^{N-1}(q^N - q^{-N}) - q^N(q^{N-1} - q^{-N+1})}{q-q^{-1}} \right) = [N]
	\end{align*}thereby illustrating one instance of invariance of $P_N$ under the first Reidemeister move. 
\end{example}

The construction of $\sl(N)$ homology, which categorifies the $\sl(N)$ polynomial, begins by categorifying the MOY polynomial of an $\sl(N)$ web. Associated to each web $W$ is a \textit{state space} $C_N(W)$, which takes the form of a finitely-generated free abelian group equipped with a $\Z$-grading called the $q$-grading. The Poincar\'e polynomial of $C_N(W)$ is $P_N(W)$. Of course, the polynomial $P_N(W)$ therefore determines the graded isomorphism type of $C_N(W)$, so as a group by itself, the state space $C_N(W)$ is contains nothing new. What makes $C_N(W)$ a categorification of $P_N(W)$ is \textit{functoriality}. There is a notion of a cobordism between webs called a \textit{foam}, and a foam $F\colon W_0 \to W_1$ induces a map $C_N(F)\colon C_N(W_0) \to C_N(W_1)$ on state spaces which is functorial under gluing. Each foam $F$ has a \textit{degree} $d(F)$, and the induced map is homogeneous of degree $d(F)$.

Just as a web is a $1$-complex in $\R^2$ with trivalent singularities and labels on its edges, a foam is a $2$-complex in $[0,1]\x\R^2$ with constrained singularities and labels on its $2$-dimensional facets. For example, the \textit{identity foam} $\Id_W$ of a web $W$ is just the product $[0,1] \x W \subset [0,1] \x \R^2$ whose additional pieces of data, such as orientations and labels, are inherited from those of $W$. One novel feature of foams is that facets may carry finitely many marked points called \textit{dots}, which are directly analogous to the dots that cobordisms may carry in Khovanov homology. For a precise definition of a foam, see for example \cite{MR4164001,https://doi.org/10.48550/arxiv.2211.08409}. 

\begin{example}\label{example:stateSpaceUnknot}
	Let $O$ denote an oriented planar loop labeled $1$ thought of as a web. Since $P_N(O) = [N]$, we know that $C_N(O)$ is graded-isomorphic to $q^{1-N}H^*(\CP^{N-1})$ where $q^{1-N}$ denotes the grading shift required so that its Poincar\'e polynomial is $[N]$. This grading shift is just the complex dimension of $\CP^{N-1}$, which is half its real dimension. It turns out that there is a particular isomorphism, which we will denote by an equality $C_N(O) = q^{1-N}H^*(\CP^{N-1})$, that is particularly nice in the following way. Consider the identity foam of $O$ with a single dot added to its interior. This foam induces a map $C_N(O) \to C_N(0)$ of degree $2$ called the \textit{dot map}. Under the identification $C_N(O) = q^{1-N}H^*(\CP^{N-1})$, the dot map is cup product with the positive generator of $q^{1-N}H^2(\CP^{N-1})$. In particular, it is cup product with the first Chern class of the hyperplane line bundle $\sr O(1)$ over $\CP^{N-1}$, which is dual to the tautological line bundle $\sr O(-1)$. The map may also be viewed as multiplication by $X$ under the standard identification $H^*(\CP^{N-1}) = \Z[X]/X^N$.
\end{example}

\begin{rem}\label{rem:transposes}
	Let $W$ be a web, and consider a foam $F\colon \emp \to W$ where $\emp$ denotes the empty web. The state space of the empty web is $\Z$, and the foam $F$ induces a map $\Z \to C_N(W)$. We may think of the foam $F$ itself as an element of the state space $C_N(W)$ by identifying it with the image of $1 \in \Z$ under its induced map. It turns out that the state space of any web is generated by foams. As an example, consider the ``cup'' foam $D\colon \emp \to O$ consisting of just a single disc. The map induced by the cup sends $1 \in \Z$ to $1 \in q^{1-N}H^0(\CP^{N-1})$. The foams obtained by adding $i$ dots to $D$ for $i = 0,1,\ldots,N-1$ are therefore a basis for the state space of $O$ by Example~\ref{example:stateSpaceUnknot}.

	Given two foams $F,G\colon \emp \to W$, there is a \textit{pairing} $\langle F,G \rangle \in \Z$ defined in the following way. The \textit{mirror} of $G$ is a foam $\ol{G}\colon W \to \emp$ obtained by reflecting $G$ through a plane. The foam obtained by gluing $F$ to $\ol{G}$ along their common boundary $W$ is a foam $F \cup_W \ol{G} \colon \emp \to \emp$. This foam induces a map $\Z \to \Z$, which is given by multiplication by an integer. This integer is precisely the pairing $\langle F,G\rangle \in \Z$. This pairing is defined on all of $C_N(W)$ and turns out to be nondegenerate. The pairing on $C_N(O) = q^{1-N}H^*(\CP^{N-1})$ is precisely the Poincar\'e pairing with respect to the negative orientation on $\CP^{N-1}$. In particular, the map $\Z \to \Z$ induced by a sphere with $N-1$ dots is multiplication by $-1$. 

	The mirror of a foam $H\colon W_0 \to W_1$ is a foam $\ol{H}\colon W_1 \to W_0$, and the induced maps $C_N(H)\colon C_N(W_0) \to C_N(W_1)$ and $C_N(\ol{H})\colon C_N(W_1)\to C_N(W_0)$ are transposes of each other with respect to the described nondegenerate pairings. 
\end{rem}

\begin{figure}[!ht]
	\centering
	\includegraphics[width=.12\textwidth]{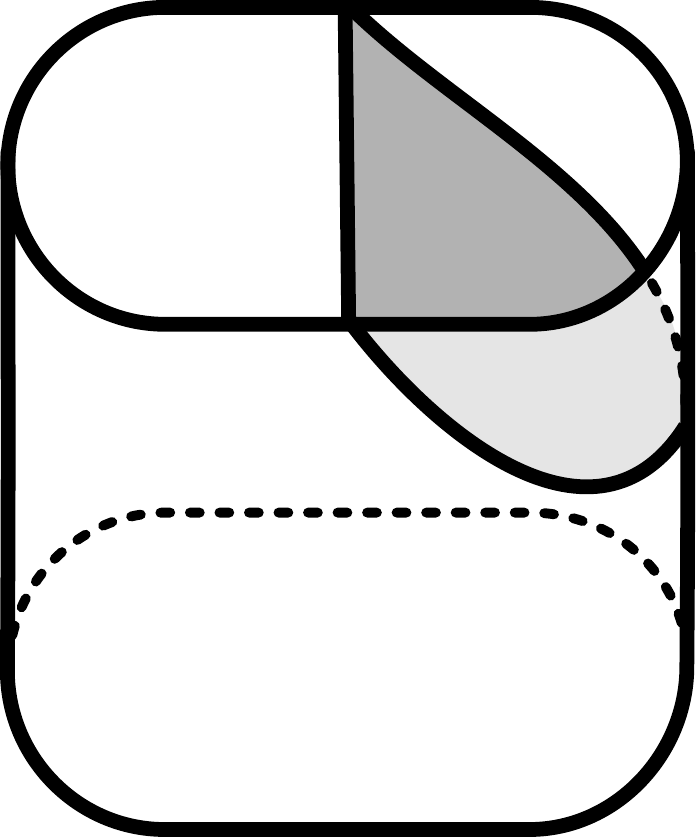}\hspace{30pt}
	\includegraphics[width=.12\textwidth]{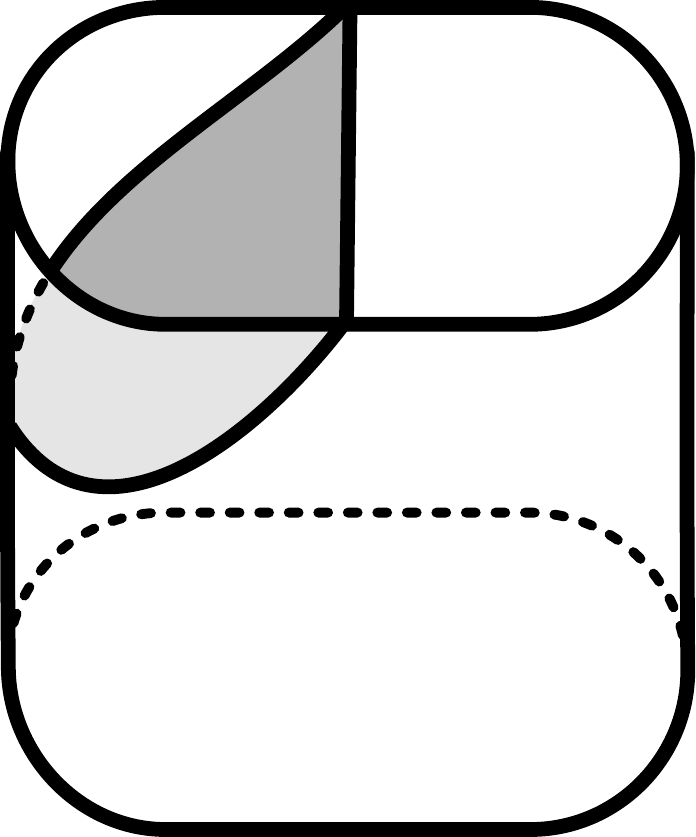}
	\captionsetup{width=.8\linewidth}
	\caption{Two foams from $O$, on the bottom, to the $\theta$-web, on the top. The maps induced on state spaces are the maps induced by the two projections $\pi_A,\pi_B\colon\F(1,1;N) \to \CP^{N-1}$.}
	\label{fig:foamsToThetaWeb}
\end{figure}

\begin{example}\label{example:stateSpaceThetaGraph}
	Consider the $\theta$-web defined in Example~\ref{example:thetaWebPolynomial} whose MOY polynomial is $[N][N-1]$. It follows that $C_N(\smallTheta)$ is graded-isomorphic to $q^{3-2N} H^*(\F(1,1;N))$. We note that the complex dimension of $\F(1,1;N)$ is $2N-3$. Again, there turns out to be a particularly nice isomorphism between them that we denote by an equality $C_N(\smallTheta) = q^{3-2N}H^*(\F(1,1;N))$. The foams in Figure~\ref{fig:foamsToThetaWeb} induce maps from $C_N(O) = q^{3-2N}H^*(\CP^{N-1})$ to $C_N(\smallTheta) = q^{3-2N} H^*(\F(1,1;N))$. These two maps are precisely the maps on cohomology induced by the two projections $\pi_A,\pi_B\colon \F(1,1;N) \to \CP^{N-1}$ given by sending $(\Lambda_A,\Lambda_B)$ to $\Lambda_A$ or $\Lambda_B$, respectively. Our convention is that the foam on the left induces $(\pi_A)^*$ while the one on the right induces $(\pi_B)^*$. 

	Just as in section~\ref{sec:SU(N)RepSpaces}, let $\sr A$ and $\sr B$ denote the tautological line bundles over $\F(1,1;N)$ whose fibers over the point $(\Lambda_A,\Lambda_B)$ are the vector spaces $\Lambda_A$ and $\Lambda_B$, respectively. Note that $\sr A$, $\sr B$ may be obtained by pulling back the tautological line bundle $\sr O(-1)$ over $\CP^{N-1}$ by the projections $\pi_A,\pi_B$. Starting with the identity foam of the $\theta$-web, consider adding a dot to either the leftmost facet or to the rightmost facet. These two foams are denoted \[
		\includegraphics[width=.11\textwidth]{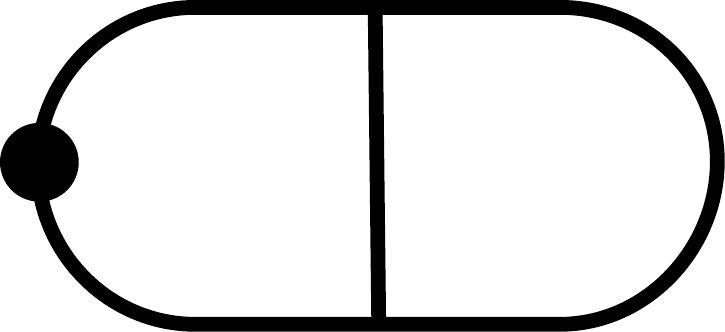} \hspace{30pt} \includegraphics[width=.11\textwidth]{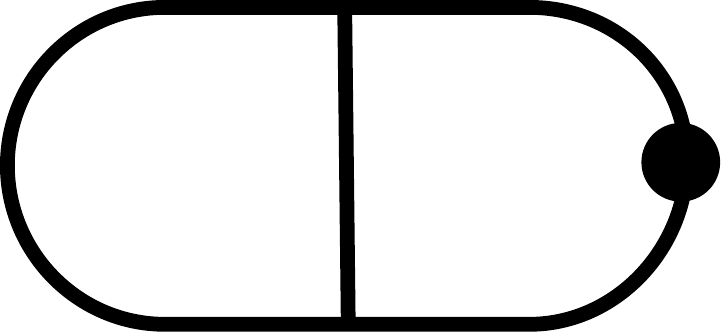}
	\]These foams induce endomorphisms of degree $2$ on the state space of the $\theta$-web. Under the identification $C_N(\smallTheta) = q^{3-2N}H^*(\F(1,1;N))$, these endomorphisms are precisely cup product with the negatives of the first Chern classes of $\sr A$ and $\sr B$, respectively. Our convention is that the foam on the left induces $-c_1(\sr A)$ while the foam on the right induces $-c_1(\sr B)$. 
\end{example}

Once the MOY polynomial of an $\sl(N)$ web has been categorified, $\sl(N)$ homology for links is obtained by categorifying the relations in Figure~\ref{fig:slNpolynomialFromMOY}. The state space of a web is thought of as a chain complex supported in homological grading zero. The complex associated to a diagram with ordered crossings is inductively defined using Figure~\ref{fig:edgeMapsZipUnzip}. In particular, the complex associated to the diagram on the left is defined to be the mapping cone of the complexes associated to the diagrams on the right. In other words, the chain complex assigned to a diagram is obtained from the cube of resolutions construction where the edge maps are given by the local foams given in Figure~\ref{fig:edgeMapsZipUnzip}. The bigraded chain homotopy type of this complex is invariant under Reidemeister moves, and its bigraded homology group is the $\sl(N)$ homology of the oriented link. 

\begin{figure}[!ht]
	\centering
	\vspace{-5pt}
	\begin{align*}
		\mathrm{C}_N\left(\:\:\begin{gathered}
			\vspace{-4pt}
			\includegraphics[width=.08\textwidth]{positiveCrossing.pdf}
		\end{gathered}\:\:\right) 
		&\coloneqq 
		\begin{tikzcd}[ampersand replacement=\&, column sep=120pt]
			h^{-1}q^{N}\mathrm{C}_N\left(\:\:\begin{gathered}
			\vspace{-4pt}
			\labellist
			\pinlabel $2$ at 130 90
			\endlabellist
			\includegraphics[width=.08\textwidth]{otherResolution}
		\end{gathered}\:\:\right) 
			\ar[r,"\displaystyle{\mathrm{C}_N\left(\:\:\:\begin{gathered}
				\vspace{-4pt}
				\includegraphics[width=.12\textwidth]{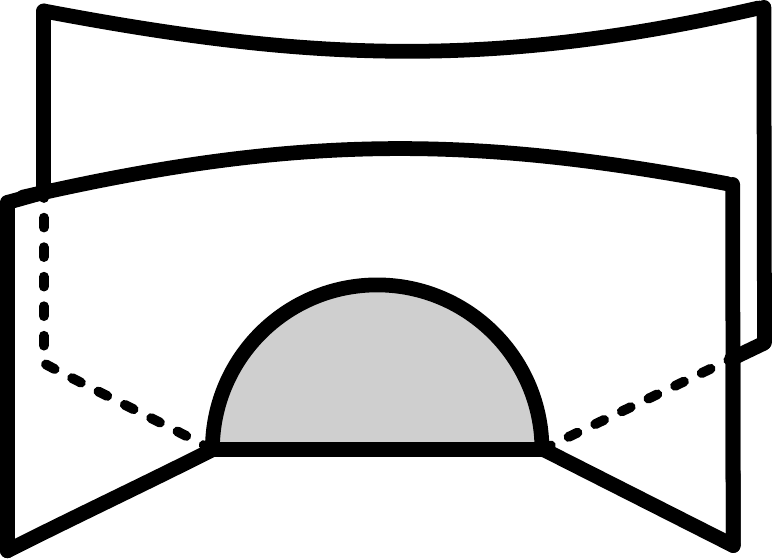}
				\end{gathered}
				\:\:\:\right)}"] \& q^{N-1}\mathrm{C}_N\left(\:\:\begin{gathered}
			\vspace{-4pt}
			\includegraphics[width=.08\textwidth]{orientedResolution.pdf}
		\end{gathered}\:\:\right)
		\end{tikzcd}\\
		\mathrm{C}_N\left(\:\:\begin{gathered}
			\vspace{-4pt}
			\includegraphics[width=.08\textwidth]{NegativeCrossing.pdf}
		\end{gathered}\:\:\right) 
		&\coloneqq 
		\begin{tikzcd}[ampersand replacement=\&, column sep=120pt]
			q^{1-N}\mathrm{C}_N\left(\:\:\begin{gathered}
			\vspace{-4pt}
			\includegraphics[width=.08\textwidth]{OrientedResolution.pdf}
		\end{gathered}\:\:\right) 
			\ar[r,"\displaystyle{\mathrm{C}_N\left(\:\:\:\begin{gathered}
				\vspace{-4pt}
				\includegraphics[width=.12\textwidth]{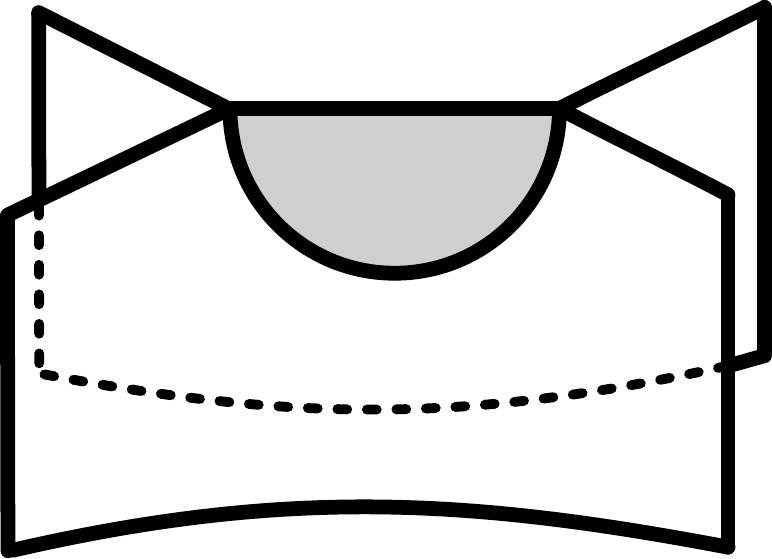}
				\end{gathered}
				\:\:\:\right)}"] \& hq^{-N}\mathrm{C}_N\left(\:\:\begin{gathered}
			\vspace{-4pt}
			\labellist
			\pinlabel $2$ at 130 90
			\endlabellist
			\includegraphics[width=.08\textwidth]{otherResolution}
		\end{gathered}\:\:\right)
		\end{tikzcd}
	\end{align*}
	\vspace{-5pt}
	\captionsetup{width=.8\linewidth}
	\caption{Edge maps in the cube of resolutions defining $\sl(N)$ homology. The notation $h^iq^j$ denotes a shift in the homological grading by $i$ and a shift in the $q$-grading by $j$. Our convention for signs in a mapping cone is that an odd shift in homological grading of a complex negates the differential.}
	\label{fig:edgeMapsZipUnzip}
\end{figure}

\begin{example}
	Consider the diagram of the unknot with a single negative crossing, which we can think of as $T(2,-1)$. Then $C_N(T(2,-1))$ is supported in homological gradings $0$ and $1$. In homological grading $0$, the complex is $q^{1-N}C_N(T(2,0)) = q^{1-N}(q^{1-N}H^*(\CP^{N-1}) \otimes q^{1-N}H^*(\CP^{N-1}))$ while in homological grading $1$, the complex is $q^{-N}C_N(\smallTheta) = q^{-N}q^{3-2N}H^*(\F(1,1;N))$. The differential is precisely the map on cohomology induced by $(\pi_A,\pi_B)\colon \F(1,1;N) \to \CP^{N-1} \x \CP^{N-1}$. This can be seen by comparing the foam defining the differential with the foams in Figure~\ref{fig:foamsToThetaWeb} together with the facts asserted in Remark~\ref{rem:transposes} and Example~\ref{example:stateSpaceThetaGraph}. 

	This map is surjective because the first Chern classes of the line bundles $\sr A$ and $\sr B$ generate $H^*(\F(1,1;N))$. It follows that the homology of this complex is supported in homological grading zero, and it is free abelian with graded rank $[N]$ thereby illustrating one instance of invariance of $\sl(N)$ homology under the first Reidemeister move. 

	Now consider the complex associated to the diagram $T(2,1)$. Then $C_N(T(2,1))$ is supported in homological gradings $-1$ and $0$. By Remark~\ref{rem:transposes}, the differential is just the transpose of the differential in the complex $C_N(T(2,-1))$ so the map is injective and its cokernel is free abelian of graded rank $[N]$. 
\end{example}

Before turning to the proof of Observation~\ref{obs:mainObs}, we record two facts that we will use. The first is a special case of the \textit{dot migration relation} \cite[Equation (3.9)]{MR3545951}, and the second is an isomorphism that categorifies the first equality in the second row of Figure~\ref{fig:MOYcalculus}, and is an analogue of delooping in Khovanov homology \cite{MR2320156}. The second follows from \cite[Equations 3.10 and 3.11]{MR3545951}.

\begin{lem}\label{lem:dotmigration}
	Generalizing the notation used in Example~\ref{example:stateSpaceThetaGraph}, we let a diagram of a web $W$ with a dot on an edge denote the identity foam of $W$ with a dot on the facet corresponding to the edge. We have the equality of maps:\[
		C_N\left(\:\begin{gathered}
			\vspace{-3pt}
			\labellist
			\pinlabel $2$ at 125 20
			\endlabellist
			\includegraphics[width=.08\textwidth]{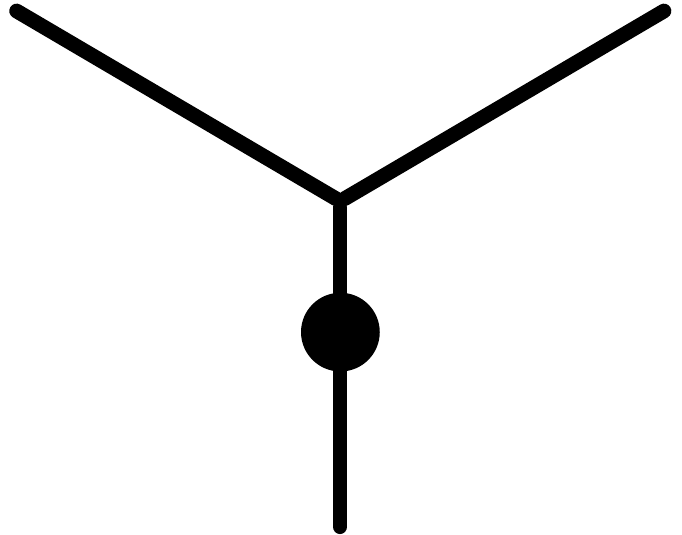}
		\end{gathered}\:\right) = C_N\left(\:\begin{gathered}
			\vspace{-3pt}
			\labellist
			\pinlabel $2$ at 125 20
			\endlabellist
			\includegraphics[width=.08\textwidth]{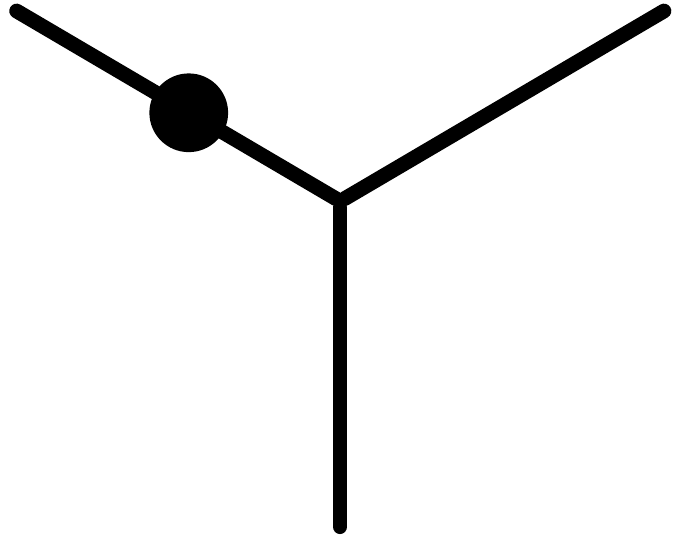}
		\end{gathered}\:\right) + C_N\left(\:\begin{gathered}
			\vspace{-3pt}
			\labellist
			\pinlabel $2$ at 125 20
			\endlabellist
			\includegraphics[width=.08\textwidth]{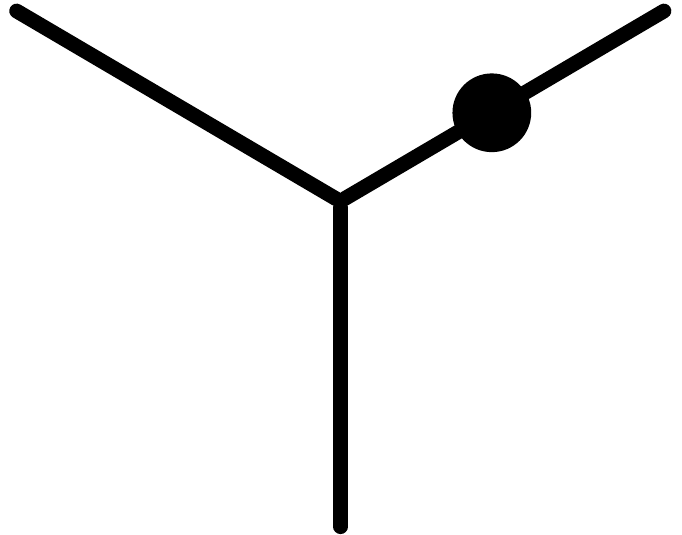}
		\end{gathered}\:\right)
	\]The relation is valid if all edges are oriented upwards or if all edges are oriented downwards. 
\end{lem}

\begin{lem}\label{lem:delooping}
	The following maps are inverse isomorphisms: \[
		\begin{tikzcd}[row sep=-25pt,column sep=100pt]
			& \:\:\:\:q C_N\left(\:\:\:\begin{gathered}
			\labellist
			\pinlabel $2$ at 40 20
			\endlabellist
			\includegraphics[width=.0085\textwidth]{verticalMOY}
		\end{gathered}\:\:\: \right)\\
		C_N\left(\:\:\begin{gathered}
			\labellist
			\pinlabel $2$ at 90 20
			\pinlabel $2$ at 90 220
			\endlabellist
			\includegraphics[width=.045\textwidth]{bigonMOY}
		\end{gathered}\:\: \right) \ar[ru,pos=0.75,"\displaystyle{C_N\left(\,\capFoam\,\right)}"] \ar[rd,swap,pos=0.85,"\displaystyle{-\,C_N\left(\,\capFoam\circ\Ldot\right)}" {yshift=1pt}] & \oplus\\
		& q^{-1} C_N\left(\:\:\:\begin{gathered}
			\labellist
			\pinlabel $2$ at 40 20
			\endlabellist
			\includegraphics[width=.0085\textwidth]{verticalMOY}
		\end{gathered}\:\:\: \right)
		\end{tikzcd}
		\hspace{10pt}
		\begin{tikzcd}[row sep=-25pt,column sep=100pt]
			\:\:\:\:q C_N\left(\:\:\:\begin{gathered}
			\labellist
			\pinlabel $2$ at 40 20
			\endlabellist
			\includegraphics[width=.0085\textwidth]{verticalMOY}
		\end{gathered}\:\:\: \right) \ar[rd,pos=0.1,"\displaystyle{C_N\left(\Rdot\circ\cupFoam\,\right)}" {yshift=-4pt}] &\\
		\oplus & C_N\left(\:\:\begin{gathered}
			\labellist
			\pinlabel $2$ at 90 20
			\pinlabel $2$ at 90 220
			\endlabellist
			\includegraphics[width=.045\textwidth]{bigonMOY}
		\end{gathered}\:\: \right)\\
		q^{-1} C_N\left(\:\:\:\begin{gathered}
			\labellist
			\pinlabel $2$ at 40 20
			\endlabellist
			\includegraphics[width=.0085\textwidth]{verticalMOY}
		\end{gathered}\:\:\: \right) \ar[ru,swap,pos=0.2,"\displaystyle{C_N\left(\,\cupFoam\,\right)}" {yshift=3pt}] &
		\end{tikzcd}
	\]In particular, we have the identities \[
		C_N\left(\:\capFoam\circ \Rdot \circ \cupFoam \:\right) = \Id = - C_N\left(\:\capFoam \circ \Ldot \circ \cupFoam \:\right)
	\]and \[
		C_N\left(\:\capFoam\circ \cupFoam \:\right) = 0 = C_N\left(\:\capFoam\circ \Ldot\circ\Rdot\circ \cupFoam \:\right).
	\]
\end{lem}

\section{Proof of Observation~\ref{obs:mainObs}}\label{sec:proofOfMainObs}

The following proposition, which is well-known to experts, is the key result on the $\sl(N)$ homology side needed to prove Observation~\ref{obs:mainObs}. See for example \cite{MR2491784,MR2948277,https://doi.org/10.48550/arxiv.1701.07525}. 

\begin{prop}\label{prop:simplifyTwists}
	For each $m \ge 1$, there is a homotopy equivalence \[
		C_N\left(\:\:\begin{gathered}
			\includegraphics[width=.035\textwidth]{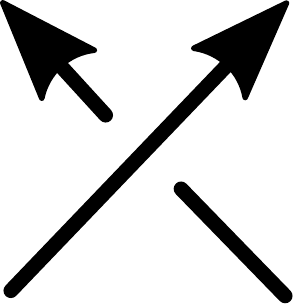}\vspace{-15pt}\\
			\vdots\vspace{-7pt}\\
			m\vspace{-10pt}\\
			\vdots\vspace{-9pt}\\
			\includegraphics[width=.035\textwidth]{smallPositiveCrossing.pdf}
			\vspace{-3pt}
		\end{gathered}\:\:\right) \simeq S_m \coloneqq \begin{tikzcd}
			S_m^{-m} \ar[r] & \cdots \ar[r] & S_m^{-2} \ar[r] & S_m^{-1} \ar[r] & S_m^0
		\end{tikzcd}
	\]where \vspace{-5pt} \[
		S^{-j}_m \coloneqq \begin{cases}
			q^{mN-m}C_N\left(\:\oriRes\:\right) & j = 0\vspace{3pt}\\
			h^{-j}q^{mN-m+ 2j - 1}C_N\left(\:\thickRes\:\right) & 0 < j \leq m
		\end{cases} \text{ and } S_m^{-j-1} \to S_m^{-j} \text{ is } \begin{cases}
			\mathrm{C}_N\left(\:\zip\:\right) & j = 0\vspace{3pt}\\
			C_N\left(\:\TRdot\:\right) - C_N\left(\:\BRdot\:\right) & j \text{ is odd}\vspace{3pt}\\
			C_N\left(\:\TRdot\:\right) - C_N\left(\:\BLdot\:\right) & j > 0 \text{ is even}.
		\end{cases}
	\]
\end{prop}

We view these pictures as local modifications of a fixed diagram, but we note that there is also a theory for tangle diagrams (see for example \cite{MR3877770}) analogous to \cite{MR2174270} for which Proposition~\ref{prop:simplifyTwists} is valid.

\begin{rem}\label{rem:simplifiedComplexes}
	Omitting the symbols $C_N(-)$ as a matter of notation, we have \vspace{-5pt} \[
		\begin{tikzcd}[column sep=50pt,row sep=small]
			q^{1-1N}S_1= &[-55pt] & & & h^{-1}q\:\thickRes \ar[r,"{\zip}"] & \oriRes\\
			q^{2-2N}S_2 = & & & h^{-2}q^3\:\thickRes \ar[r,"{\TRdot \displaystyle{-} \BRdot}" {yshift=-2pt}] & h^{-1}q\:\thickRes \ar[r,"{\zip}"] & \oriRes\\
			q^{3-3N}S_3 = & & h^{-3}q^5\:\thickRes \ar[r,"{\TRdot \displaystyle{-} \BLdot}" {yshift=-2pt}] & h^{-2}q^3\:\thickRes \ar[r,"{\TRdot \displaystyle{-} \BRdot}" {yshift=-2pt}] & h^{-1}q\:\thickRes \ar[r,"{\zip}"] & \oriRes\\
			q^{4-4N}S_4 = & h^{-4}q^7\:\thickRes \ar[r,"{\TRdot \displaystyle{-} \BRdot}" {yshift=-2pt}] & h^{-3}q^5\:\thickRes \ar[r,"{\TRdot \displaystyle{-} \BLdot}" {yshift=-2pt}] & h^{-2}q^3\:\thickRes \ar[r,"{\TRdot \displaystyle{-} \BRdot}" {yshift=-2pt}] & h^{-1}q\:\thickRes \ar[r,"{\zip}"] & \oriRes
		\end{tikzcd}
	\]
\end{rem}

\begin{proof}[Proof of Proposition~\ref{prop:simplifyTwists}]
	We omit the symbols $C_N(-)$ for notational simplicity as in Remark~\ref{rem:simplifiedComplexes}. 
	We prove the result by induction on $m \ge 1$. The base case $m = 1$ is true by definition (Figure~\ref{fig:edgeMapsZipUnzip}). For the inductive case, we must prove that \[
		q^{1-N}\left(\begin{tikzcd}[column sep=50pt]
			h^{-m}q^{2m-1}\:\begin{gathered}
			\includegraphics[width=.036\textwidth]{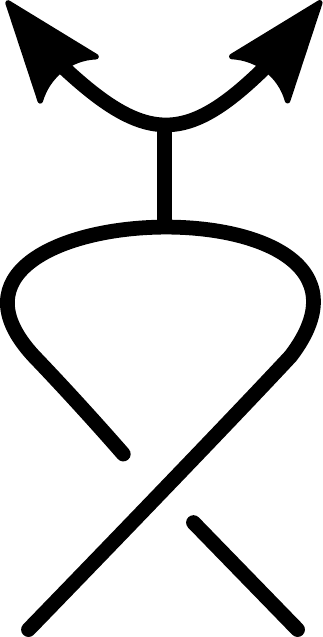}
			\vspace{-3pt}
		\end{gathered} \ar[r] & \cdots \ar[r] & h^{-1}q\:\begin{gathered}
			\includegraphics[width=.036\textwidth]{smallThickEdgeTwist.pdf}
			\vspace{-3pt}
		\end{gathered} \ar[r] & \begin{gathered}
			\includegraphics[width=.036\textwidth]{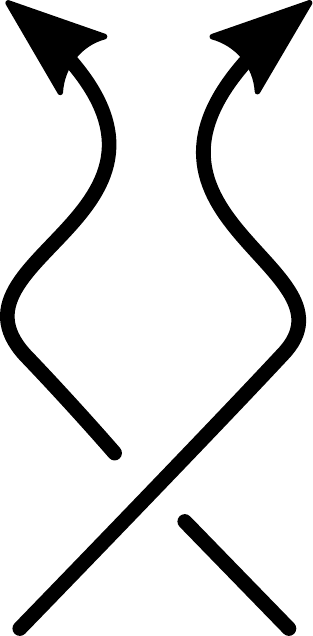}
			\vspace{-3pt}
		\end{gathered}
		\end{tikzcd} \right) \simeq \: q^{(m+1)-(m+1)N}S_{m+1}
	\]where the arrows in the complex on the lefthand side are induced by those in $S_m$. The complex expands to \[
		\begin{tikzpicture}
			\node (1t) at (0,4) {$h^{-3}q^5\:\thickRes$};
			\node (2t) at (4.5,4) {$h^{-2}q^3\:\thickRes$}; 
			\node (3t) at (9,4) {$h^{-1}q\:\thickRes$}; 
			\node (4t) at (13.5,4) {$\oriRes$}; 

			\node (1b) at (0,0) {$h^{-4}q^6\:\thickLoop$};
			\node (2b) at (4.5,0) {$h^{-3}q^4\:\thickLoop$};
			\node (3b) at (9,0) {$h^{-2}q^2\:\thickLoop$};
			\node (4b) at (13.5,0) {$h^{-1}q\:\thickRes$};

			\draw[->] (1t) to node[above]{$\TRdot - \BLdot$} (2t);
			\draw[->] (2t) to node[above]{$\TRdot - \BRdot$} (3t);
			\draw[->] (3t) to node[above]{$\zip$} (4t);

			\draw[->,neg] (1b) to node[right]{$\leftZip$} (1t);
			\draw[->] (2b) to node[right]{$\leftZip$} (2t);
			\draw[->,neg] (3b) to node[right]{$\leftZip$} (3t);
			\draw[->] (4b) to node[right]{$\zip$} (4t);

			\draw[->] (1b) to node[above]{$\loopTRdot \displaystyle{-} \loopBLdot$} (2b);
			\draw[->] (2b) to node[above]{$\loopTRdot \displaystyle{-} \loopBRdot$} (3b);
			\draw[->] (3b) to node[above]{$\rightZip$} (4b);
		\end{tikzpicture}
	\]where we have drawn the $m = 3$ case for clarity. The small circles drawn on the vertical arrows in every other column indicate negation (see the caption of Figure~\ref{fig:edgeMapsZipUnzip}). By Lemma~\ref{lem:delooping}, we have inverse isomorphisms \[
		\begin{tikzpicture}
			\node (1l) at (0,0) {$\thickLoop$};
			\node (1r1) at (5,0.6) {$\:\:\:\:q\:\thickRes$};
			\node (1rs) at (5.27,0) {$\oplus$};
			\node (1r2) at (5,-0.6) {$q^{-1}\:\thickRes$};

			\draw[->] (1l) to node[above=2pt]{$\capEx$} (1r1);
			\draw[->,neg] (1l) to node[below=4pt]{$\capEx \circ \loopBLdot$} (1r2);

			\node (2r) at (13,0) {$\thickLoop$};
			\node (2l1) at (8,0.6) {$\:\:\:\:q\:\thickRes$};
			\node (2ls) at (8.27,0) {$\oplus$};
			\node (2l2) at (8,-0.6) {$q^{-1}\:\thickRes$};

			\draw[->] (2l1) to node[above=4pt]{$\loopBRdot\circ\cupEx$} (2r);
			\draw[->] (2l2) to node[below=2pt]{$\cupEx$} (2r);
		\end{tikzpicture}
	\]Using this isomorphism, we may rewrite the complex as \[
		\begin{tikzpicture}
			\node (1t) at (0,3) {$h^{-3}q^5\:\thickRes$};
			\node (2t) at (4.5,3) {$h^{-2}q^3\:\thickRes$}; 
			\node (3t) at (9,3) {$h^{-1}q\:\thickRes$}; 
			\node (4t) at (13.5,3) {$\oriRes$}; 

			\node (1bp) at (0,0) {$\oplus$};
			\node (1b1) at (0.7,0.6) {$h^{-4}q^7\:\thickRes$};
			\node (1b2) at (-0.7,-0.6) {$h^{-4}q^5\:\thickRes$};

			\node (2bp) at (4.5,0) {$\oplus$};
			\node (2b1) at (4.5+0.7,0.6) {$h^{-3}q^5\:\thickRes$};
			\node (2b2) at (4.5-0.7,-0.6) {$h^{-3}q^3\:\thickRes$};

			\node (3bp) at (9,0) {$\oplus$};
			\node (3b1) at (9+0.7,0.6) {$h^{-2}q^3\:\thickRes$};
			\node (3b2) at (9-0.7,-0.6) {$h^{-2}q\:\thickRes$};

			\node (4b) at (13.5,0) {$h^{-1}q\:\thickRes$};

			\draw[->] (1t) to node[above]{$\TRdot - \BLdot$} (2t);
			\draw[->] (2t) to node[above]{$\TRdot - \BRdot$} (3t);
			\draw[->] (3t) to node[above]{$\zip$} (4t);

			\draw[->,dotted,neg] (1b1) to node{} (1t);
			\draw[->,dotted,neg] (1b2) to node{} (1t);

			\draw[->,dotted] (2b1) to node{} (2t);
			\draw[->,dotted] (2b2) to node{} (2t);

			\draw[->,dotted,neg] (3b1) to node{} (3t);
			\draw[->,dotted,neg] (3b2) to node{} (3t);

			\draw[->] (4b) to node[right]{$\zip$} (4t);

			\draw[->,dotted] (1b1) to node{} (2b1);
			\draw[->,dotted] (1b1) to node{} (2b2);
			\draw[->,dotted] (1b2) to node{} (2b1);
			\draw[->,dotted] (1b2) to node{} (2b2);

			\draw[->,dotted] (2b1) to node{} (3b1);
			\draw[->,dotted] (2b1) to node{} (3b2);
			\draw[->,dotted] (2b2) to node{} (3b1);
			\draw[->,dotted] (2b2) to node{} (3b2);

			\draw[->,dotted] (3b1) to node{} (4b);
			\draw[->,dotted] (3b2) to node{} (4b);
		\end{tikzpicture}
	\]where the dotted arrows are composites, which we now consider. The vertical arrows can be determined using the observation that \[
		\begin{gathered}
			\includegraphics[width=.24\textwidth]{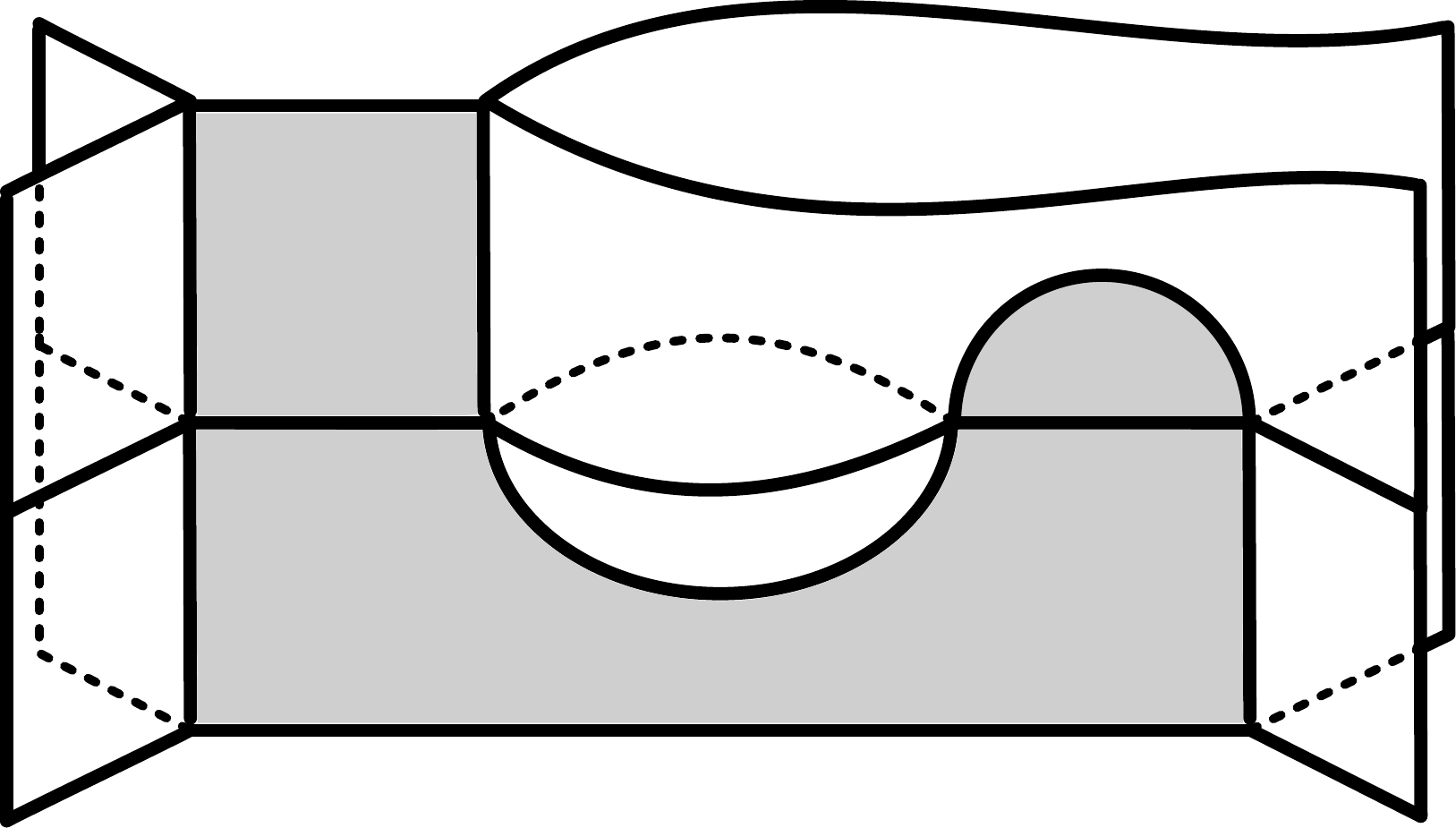}
		\end{gathered}
	\]is isotopic to the identity foam. A similar observation allows us to also determine the rightmost horizontal arrows. We therefore have \[
		\begin{tikzpicture}
			\node (1t) at (0,3) {$h^{-3}q^5\:\thickRes$};
			\node (2t) at (4.5,3) {$h^{-2}q^3\:\thickRes$}; 
			\node (3t) at (9,3) {$h^{-1}q\:\thickRes$}; 
			\node (4t) at (13.5,3) {$\oriRes$}; 

			\node (1bp) at (0,0) {$\oplus$};
			\node (1b1) at (0.7,0.6) {$h^{-4}q^7\:\thickRes$};
			\node (1b2) at (-0.7,-0.6) {$h^{-4}q^5\:\thickRes$};

			\node (2bp) at (4.5,0) {$\oplus$};
			\node (2b1) at (4.5+0.7,0.6) {$h^{-3}q^5\:\thickRes$};
			\node (2b2) at (4.5-0.7,-0.6) {$h^{-3}q^3\:\thickRes$};

			\node (3bp) at (9,0) {$\oplus$};
			\node (3b1) at (9+0.7,0.6) {$h^{-2}q^3\:\thickRes$};
			\node (3b2) at (9-0.7,-0.6) {$h^{-2}q\:\thickRes$};

			\node (4b) at (13.5,0) {$h^{-1}q\:\thickRes$};

			\draw[->] (1t) to node[above]{$\TRdot - \BLdot$} (2t);
			\draw[->] (2t) to node[above]{$\TRdot - \BRdot$} (3t);
			\draw[->] (3t) to node[above]{$\zip$} (4t);

			\draw[->,neg] (1b1) to node[right=2pt]{$\BRdot$} (1t);
			\draw[->,neg,pos=.7] (1b2) to node[left]{$\Id$} (1t);

			\draw[->] (2b1) to node[right=2pt]{$\BRdot$} (2t);
			\draw[->,pos=.7] (2b2) to node[left]{$\Id$} (2t);

			\draw[->,neg] (3b1) to node[right=2pt]{$\BRdot$} (3t);
			\draw[->,neg,pos=.7] (3b2) to node[left]{$\Id$} (3t);

			\draw[->] (4b) to node[right]{$\zip$} (4t);

			\draw[->,dotted] (1b1) to node{} (2b1);
			\draw[->,dotted] (1b1) to node{} (2b2);
			\draw[->,dotted] (1b2) to node{} (2b1);
			\draw[->,dotted] (1b2) to node{} (2b2);

			\draw[->,dotted] (2b1) to node{} (3b1);
			\draw[->,dotted] (2b1) to node{} (3b2);
			\draw[->,dotted] (2b2) to node{} (3b1);
			\draw[->,dotted] (2b2) to node{} (3b2);

			\draw[->] (3b1) to node[above]{$\TRdot$} (4b);
			\draw[->] (3b2) to node[below]{$\Id$} (4b);
		\end{tikzpicture}
	\]We may therefore simplify the complex by Gaussian elimination along the $m \:(=3)$ vertical identity maps (for Gaussian elimination, see for example \cite[Lemma 4.2]{MR2320156} or \cite[Lemma 3.1]{MR3332892}). Only half of the remaining dotted maps need to be computed explicitly to determine the result of Gaussian elimination. They are \[
		\begin{tikzpicture}
			\node (1t) at (0,3) {$h^{-3}q^5\:\thickRes$};
			\node (2t) at (4.5,3) {$h^{-2}q^3\:\thickRes$}; 
			\node (3t) at (9,3) {$h^{-1}q\:\thickRes$}; 
			\node (4t) at (13.5,3) {$\oriRes$}; 

			\node (1bp) at (0,0) {$\oplus$};
			\node (1b1) at (0.7,0.6) {$h^{-4}q^7\:\thickRes$};
			\node (1b2) at (-0.7,-0.6) {$h^{-4}q^5\:\thickRes$};

			\node (2bp) at (4.5,0) {$\oplus$};
			\node (2b1) at (4.5+0.7,0.6) {$h^{-3}q^5\:\thickRes$};
			\node (2b2) at (4.5-0.7,-0.6) {$h^{-3}q^3\:\thickRes$};

			\node (3bp) at (9,0) {$\oplus$};
			\node (3b1) at (9+0.7,0.6) {$h^{-2}q^3\:\thickRes$};
			\node (3b2) at (9-0.7,-0.6) {$h^{-2}q\:\thickRes$};

			\node (4b) at (13.5,0) {$h^{-1}q\:\thickRes$};

			\draw[->] (1t) to node[above]{$\TRdot - \BLdot$} (2t);
			\draw[->] (2t) to node[above]{$\TRdot - \BRdot$} (3t);
			\draw[->] (3t) to node[above]{$\zip$} (4t);

			\draw[->,neg] (1b1) to node[right=2pt]{$\BRdot$} (1t);
			\draw[->,neg,pos=.7] (1b2) to node[left]{$\Id$} (1t);

			\draw[->] (2b1) to node[right=2pt]{$\BRdot$} (2t);
			\draw[->,pos=.7] (2b2) to node[left]{$\Id$} (2t);

			\draw[->,neg] (3b1) to node[right=2pt]{$\BRdot$} (3t);
			\draw[->,neg,pos=.7] (3b2) to node[left]{$\Id$} (3t);

			\draw[->] (4b) to node[right]{$\zip$} (4t);

			\draw[->] (1b1) to node[above]{$\TRdot$} (2b1);
			\draw[->,dotted] (1b1) to node{} (2b2);
			\draw[->] (1b2) to node[below]{$\Id$} (2b1);
			\draw[->,dotted] (1b2) to node{} (2b2);

			\draw[->,neg] (2b1) to node[above]{$\TLdot$} (3b1);
			\draw[->,dotted] (2b1) to node{} (3b2);
			\draw[->,neg] (2b2) to node[below]{$\Id$} (3b1);
			\draw[->,dotted] (2b2) to node{} (3b2);

			\draw[->] (3b1) to node[above]{$\TRdot$} (4b);
			\draw[->] (3b2) to node[below]{$\Id$} (4b);
		\end{tikzpicture}
	\]To be explicit, we are asserting that \begin{align*}
		\Id &=\: \capEx \circ \left(\:\loopTRdot - \loopBLdot \:\right) \circ \cupEx\\
		\TRdot \:&=\:\capEx\circ\left(\:\loopTRdot - \loopBLdot \:\right)\circ \loopBRdot \circ\cupEx\\
		-\Id &=\:\capEx\circ\left(\:\loopTRdot - \loopBRdot \:\right) \circ\cupEx\\
		-\TLdot \:&= \:\capEx\circ\left(\:\loopTRdot - \loopBRdot \:\right)\circ \loopBRdot \circ \cupEx
	\end{align*}The first three computations follow from Lemma~\ref{lem:delooping}, while the last computation uses the identity \[
		\TLdot + \TRdot = \BLdot + \BRdot
	\]which follows from Lemma~\ref{lem:dotmigration}. Finally, we perform the stated Gaussian elimination to obtain\[
		\begin{tikzpicture} 
			\node (4t) at (13.5,3) {$\oriRes$}; 

			\node (1b1) at (0,0) {$h^{-4}q^7\:\thickRes$};

			\node (2b1) at (4.5,0) {$h^{-3}q^5\:\thickRes$};

			\node (3b1) at (9,0) {$h^{-2}q^3\:\thickRes$};

			\node (4b) at (13.5,0) {$h^{-1}q\:\thickRes$};

			\draw[->] (4b) to node[right]{$\zip$} (4t);

			\draw[->] (1b1) to node[above]{$\TRdot-\BRdot$} (2b1);

			\draw[->] (2b1) to node[above]{$\TRdot - \BLdot$} (3b1);

			\draw[->] (3b1) to node[above]{$\TRdot-\BRdot$} (4b);
		\end{tikzpicture}
	\]which is precisely $q^{(m+1)-(m+1)N}S_{m+1}$ as required.
\end{proof}

\begin{proof}[Proof of Observation~\ref{obs:mainObs}]
	Let $O$ denote the crossingless diagram of the unknot, and note that when $m = 0$, the torus link $T(2,0)$ is just the two-component unlink $O \sqcup O$. Since $\sr R_N(O \sqcup O) = \CP^{N-1} \x \CP^{N-1}$, the $m = 0$ case of the observation follows from $\KR_N(O) \cong q^{1-N}H^*(\CP^{N-1})$ using $\KR_N(O \sqcup O) \cong \KR_N(O) \otimes \KR_N(O)$ and $H^*(\CP^{N-1} \x \CP^{N-1}) \cong H^*(\CP^{N-1}) \otimes H^*(\CP^{N-1})$. We may assume that $m \ge 1$, since $\sr R_N(T(2,-m))$ and $\sr R_N(T(2,m))$ are the same, and $\KR_N(T(2,-m))$ and $\KR_N(T(2,m))$ are related by the universal coefficient theorem, which can be shown using Remark~\ref{rem:transposes}. 

	Proposition~\ref{prop:simplifyTwists} gives a homotopy equivalence between $C_N(T(2,m))$ and a much simpler complex. For the first few values of $m$, for example, we form the braid closures of the local diagrams in Remark~\ref{rem:simplifiedComplexes} to obtain \[
		\begin{tikzcd}[column sep=50pt,row sep=tiny]
			q^{1-1N}C_N T(2,1) \simeq &[-55pt] &[-40pt] &[30pt] &[-40pt] q^{1-N}C_N T(2,1)\\
			q^{2-2N}C_N T(2,2) \simeq &[-55pt] & & h^{-2}q^3C_N\smallTheta & q^{1-N}C_N T(2,1)\\
			q^{3-3N}C_N T(2,3) \simeq & & h^{-3}q^5C_N\smallTheta \ar[r,"{C_N\smallThetaR\, - C_N\smallThetaL}"] & h^{-2}q^3C_N\smallTheta & q^{1-N}C_N T(2,1)\\
			q^{4-4N}C_N T(2,4) \simeq & h^{-4}q^7C_N\smallTheta & h^{-3}q^5C_N\smallTheta \ar[r,"{C_N\smallThetaR\, - C_N\smallThetaL}"] & h^{-2}q^3C_N\smallTheta & q^{1-N} C_N T(2,1)
		\end{tikzcd}
	\]Note that the maps from even homological degree to odd homological degree are all zero. Furthermore, since $T(2,1)$ represents the unknot, $C_N(T(2,1))$ is homotopy equivalent to $C_N(O)$. Thus, if we set \[
		A \coloneqq \begin{tikzcd}[column sep=80pt]
			h^{-1}q C_N\smallTheta \ar[r,"{C_N\smallThetaR\, - C_N\smallThetaL}"] & q^{-1}C_N\smallTheta
		\end{tikzcd}
	\]then \[
		 q^{m-mN}C_N(T(2,m)) \simeq \begin{cases}
		 	\left(\bigoplus_{i=1}^{k}h^{-2i}q^{4i} A \right) \oplus \left(q^{1-N} C_N(O)\right)  & m\text{ is odd and }m = 2k + 1\\
		 	\left(h^{-2k}q^{4k-1}C_N(\smallTheta)\right) \oplus\left(\bigoplus_{i=1}^{k-1} h^{-2i}q^{4i}A \right) \oplus \left(q^{1-N} C_N(O)\right)& m\text{ is even and }m=2k.
		 \end{cases}
	\]Proposition~\ref{prop:representationSpaceCalculation} gives an explicit calculation of $\sr R_N(T(2,m))$: it consists of a single copy of $\CP^{N-1}$, a single copy of $\F(1,1;N)$ when $m$ is even, and $\lfloor (m-1)/2 \rfloor$ copies of $UT\CP^{N-1}$. We recall from Examples~\ref{example:stateSpaceUnknot} and \ref{example:stateSpaceThetaGraph} that $C_N(O) = q^{1-N}H^{*}(\CP^{N-1})$ and $C_N(\smallTheta) = q^{3-2N} H^*(\F(1,1;N))$. Thus, in order to show that $\KR_N(T(2,m))$ and $H^*(\sr R_N(T(2,m)))$ are isomorphic as abelian groups, it suffices to show that the homology of the complex $A$ is isomorphic to $H^*(UT\CP^{N-1})$. 

	By Proposition~\ref{prop:eulerClass}, we know that $UT\CP^{N-1}$ is a circle bundle over $\F(1,1;N)$ with Euler class $c_1(\sr A) - c_1(\sr B)$, where $\sr A$ and $\sr B$ are the two tautological line bundles over $\F(1,1;N)$. Since the cohomology of $\F(1,1;N)$ is supported in even degrees, the Gysin sequence for this bundle is the exact sequence \[
		\begin{tikzcd}
			0 \ar[r] &[-10pt] H^{2i+1}(UT\CP^{N-1}) \ar[r] & H^{2i}(\F(1,1;N)) \ar[r,"\cup\,e"] & H^{2i+2}(\F(1,1;N)) \ar[r] & H^{2i+2}(UT\CP^{N-1}) \ar[r] &[-10pt] 0
		\end{tikzcd}
	\]Under the identification $C_N(\smallTheta) = q^{3-2N}H^*(\F(1,1;N))$ given in Example~\ref{example:stateSpaceThetaGraph}, the complex $A$ is precisely \[
		\begin{tikzcd}
			h^{-1}q^{4-2N}H^*(\F(1,1;N)) \ar[r,"\cup\,e"] & q^{2-2N}H^*(\F(1,1;N))
		\end{tikzcd}
	\]from which it immediately follows that the homology of $A$ is isomorphic to the cohomology of $UT\CP^{N-1}$, without needing to actually compute either group. 

	Notice that we may actually describe the homology of $A$ as a bigraded group in terms of the cohomology of $UT\CP^{N-1}$, which thereby gives the complete bigrading information on $\KR_N(T(2,m))$. The homology of $A$ is supported in homological degrees $-1$ and $0$, and in even $q$-gradings. In homological degree $0$, the $q$-graded homology of $A$ is simply the direct sum of the even-degree cohomology groups of $UT\CP^{N-1}$, shifted in grading by $2-2N$. In homological degree $-1$, the $q$-graded homology of $A$ is the direct sum of the odd-degree cohomology groups of $UT\CP^{N-1}$, shifted in grading by $3-2N$. 

	We can explicitly determine the cohomology of $UT\CP^{N-1}$ using the Gysin sequence for the $(2N-3)$-sphere bundle $UT\CP^{N-1} \to \CP^{N-1}$ whose Euler class is $N = \chi(\CP^{N-1})$ times a generator of $H^{2N-2}(\CP^{N-1})$. We find that \[
		H^i(UT\CP^{N-1}) = \begin{cases}
			\Z & i \text{ is even and } 0 \leq i \leq 2N-4\\
			\Z/N & i = 2N-2\\
			\Z & i\text{ is odd and } 2N - 1 \leq i \leq 4N-5\\
			0 & \text{else.} 
		\end{cases}
	\]Thus, in homological degree $0$, the homology of $A$ consists of a copy of $\Z$ in even $q$-gradings $i$ satisfying $2-2N \leq i \leq -2$ and a copy of $\Z/N$ in $q$-grading $0$. In homological degree $-1$, the homology of $A$ consists of a copy of $\Z$ in even $q$-gradings $i$ satisfying $2 \leq i \leq 2N-2$.
\end{proof}

\raggedright
\bibliography{gysin}

\begin{thebibliography}{MOY98}

\bibitem[BN05]{MR2174270}
Dror Bar-Natan.
\newblock Khovanov's homology for tangles and cobordisms.
\newblock {\em Geom. Topol.}, 9:1443--1499, 2005.

\bibitem[BN07]{MR2320156}
Dror Bar-Natan.
\newblock Fast {K}hovanov homology computations.
\newblock {\em J. Knot Theory Ramifications}, 16(3):243--255, 2007.

\bibitem[BS15]{MR3332892}
Joshua Batson and Cotton Seed.
\newblock A link-splitting spectral sequence in {K}hovanov homology.
\newblock {\em Duke Math. J.}, 164(5):801--841, 2015.

\bibitem[ETW18]{MR3877770}
Michael Ehrig, Daniel Tubbenhauer, and Paul Wedrich.
\newblock Functoriality of colored link homologies.
\newblock {\em Proc. Lond. Math. Soc. (3)}, 117(5):996--1040, 2018.

\bibitem[Gra13]{MR3125899}
Jonathan Grant.
\newblock The moduli problem of {L}obb and {Z}entner and the colored
  {$\fk{sl}(N)$} graph invariant.
\newblock {\em J. Knot Theory Ramifications}, 22(10):1350060, 16, 2013.

\bibitem[KM11]{MR2860345}
P.~B. Kronheimer and T.~S. Mrowka.
\newblock Knot homology groups from instantons.
\newblock {\em J. Topol.}, 4(4):835--918, 2011.

\bibitem[KR08]{MR2391017}
Mikhail Khovanov and Lev Rozansky.
\newblock Matrix factorizations and link homology.
\newblock {\em Fund. Math.}, 199(1):1--91, 2008.

\bibitem[Kra09]{MR2491784}
Daniel Krasner.
\newblock A computation in {K}hovanov-{R}ozansky homology.
\newblock {\em Fund. Math.}, 203(1):75--95, 2009.

\bibitem[Lam75]{MR431194}
Kee~Yuen Lam.
\newblock A formula for the tangent bundle of flag manifolds and related
  manifolds.
\newblock {\em Trans. Amer. Math. Soc.}, 213:305--314, 1975.

\bibitem[LZ14]{MR3190356}
Andrew Lobb and Raphael Zentner.
\newblock The quantum {${\rm sl}(N)$} graph invariant and a moduli space.
\newblock {\em Int. Math. Res. Not. IMRN}, (7):1956--1972, 2014.

\bibitem[MOY98]{MR1659228}
Hitoshi Murakami, Tomotada Ohtsuki, and Shuji Yamada.
\newblock {HOMFLY} polynomial via an invariant of colored plane graphs.
\newblock {\em Enseign. Math. (2)}, 44(3-4):325--360, 1998.

\bibitem[MSV09]{MR2491657}
Marco Mackaay, Marko Sto\v{s}i\'{c}, and Pedro Vaz.
\newblock {$\fk{sl}(N)$}-link homology {$(N\geq 4)$} using foams and the
  {K}apustin-{L}i formula.
\newblock {\em Geom. Topol.}, 13(2):1075--1128, 2009.

\bibitem[QR16]{MR3545951}
Hoel Queffelec and David E.~V. Rose.
\newblock The {$\fk{sl}_n$} foam 2-category: a combinatorial formulation of
  {K}hovanov--{R}ozansky homology via categorical skew {H}owe duality.
\newblock {\em Adv. Math.}, 302:1251--1339, 2016.

\bibitem[Ras07]{MR2309174}
Jacob Rasmussen.
\newblock Khovanov-{R}ozansky homology of two-bridge knots and links.
\newblock {\em Duke Math. J.}, 136(3):551--583, 2007.

\bibitem[RW20]{MR4164001}
Louis-Hadrien Robert and Emmanuel Wagner.
\newblock A closed formula for the evaluation of foams.
\newblock {\em Quantum Topol.}, 11(3):411--487, 2020.

\bibitem[Tho17]{https://doi.org/10.48550/arxiv.1701.07525}
Benjamin Thompson.
\newblock Khovanov complexes of rational tangles.
\newblock arXiv:1701.07525, 2017.

\bibitem[Wan21]{wang2021sln}
Joshua Wang.
\newblock On sl({N}) link homology with mod {N} coefficients.
\newblock arXiv:2111.02287, 2021.

\bibitem[Wan22]{https://doi.org/10.48550/arxiv.2211.08409}
Joshua Wang.
\newblock Colored {sl(N)} homology and {SU(N)} representations of the trefoil.
\newblock arXiv:2211.08409, 2022.

\bibitem[Wu12]{MR2948277}
Hao Wu.
\newblock Simplified {K}hovanov-{R}ozansky chain complexes of open 2-braids.
\newblock {\em Topology Appl.}, 159(14):3190--3203, 2012.

\end{thebibliography}
\bibliographystyle{alpha}

\vspace{10pt}

\textit{Department of Mathematics}

\textit{Harvard University}

\textit{Science Center, 1 Oxford Street}

\textit{Cambridge, MA 02138}

\textit{USA}

\vspace{10pt}

\textit{Email:} \texttt{jxwang@math.harvard.edu}

\end{document}